\documentclass{article}

\usepackage{amssymb}
\usepackage{amsmath}
\usepackage{bbm}

\usepackage{amsthm}
\newtheorem{lemma}{Lemma}

\newtheorem{proposition}{Proposition}

\theoremstyle{definition}
\newtheorem{example}{Example}

\usepackage[noblocks]{authblk}

\usepackage{graphicx}
\usepackage[captionskip=3mm]{subfig}
\usepackage{epstopdf}
\usepackage{float}
\graphicspath{{figures/}}

\usepackage{booktabs}
\usepackage{microtype}

\usepackage{placeins}

\usepackage[hyperfootnotes=false]{hyperref}

\DeclareMathOperator{\E}{\mathcal{E}}

\DeclareMathOperator{\supp}{\operatorname{supp}}
\renewcommand{\H}{\mathcal{H}}
\renewcommand{\div}{\operatorname{div}}
\newcommand{\norm}[1]{\left\| #1 \right\|}

\numberwithin{equation}{section}
\begin{document}

\title{Quantitative photoacoustic tomography with piecewise constant material parameters}

\date{\today}

\author[1]{W.~Naetar ({wolf.naetar@univie.ac.at})}

\author[1,2]{O.~Scherzer ({otmar.scherzer@univie.ac.at})}

\affil[1]{\footnotesize Computational Science Center, University of Vienna \authorcr Oskar Morgenstern-Platz 1, A-1090 Vienna, Austria \vspace{.5\baselineskip}}
\affil[2]{\footnotesize Radon Institute of Computational and Applied Mathematics, Austrian Academy of Sciences \authorcr Altenbergerstr. 69, A-4040 Linz, Austria \vspace{.5\baselineskip}}

\maketitle

\section*{Abstract}
The goal of \emph{quantitative} photoacoustic tomography is to determine optical and acoustical material properties from initial pressure maps as obtained, 
for instance, from photoacoustic imaging. The most relevant parameters are absorption, diffusion and Gr\"{u}neisen coefficients, all of which can be heterogeneous. 
Recent work by Bal and Ren shows that in general, unique reconstruction of all three parameters is impossible, even if multiple measurements of the initial pressure (corresponding to different laser excitation directions at a single wavelength) are available. 

Here, we propose a restriction to piecewise constant material parameters. We show that in the diffusion approximation of light transfer, piecewise constant absorption, diffusion \emph{and} 
Gr\"{u}neisen coefficients can be recovered uniquely from photoacoustic measurements at a single wavelength. In addition, we implemented our ideas numerically and tested them on simulated three-dimensional data.

\paragraph*{Keywords. Quantitative photoacoustic tomography, mathematical imaging, inverse problems} 

\paragraph*{AMS subject classifications.}  35R25, 35R30, 65J22, 92C55

\section{Introduction}
\label{sec:intro}
Photoacoustic tomography (PAT) is a hybrid imaging technique utilizing the coupling of laser excitations with ultrasound measurements. Tissue irradiated by a short monochromatic laser pulse generates an ultrasound signal (due to thermal expansion) which can be measured by ultrasound transducers outside the medium. From these measurements, the ultrasound wave's initial pressure (whose spatial variation depends on material properties of the tissue) can be reconstructed uniquely by solving a well-studied inverse problem for the wave equation. For further information on this inverse problem, see, e.g., Kuchment and Kunyansky \cite{KucKun08}.

The obtained ultrasound initial pressure \emph{qualitatively} resembles the structure of the tissue (i.e., its inhomogeneities are visible). It is, however, desirable to image material parameters (whose values can serve as diagnostic information) instead. That is the goal of \emph{quantitative} photoacoustic tomography (qPAT).

Mathematically, the problem can be posed as follows. In biological tissue, where photon scattering is a dominant effect compared to absorption, light transfer can be described by the \emph{diffusion approximation} of the \emph{radiative transfer equation}. It is valid in regions $\Omega \subset \mathbb{R}^3$ with sufficient distance to the light source and is given by
\begin{equation}
\label{eq:diff_eq}
	-\div(D(x) \nabla u(x)) + \mu(x) u(x) = 0. 
\end{equation}

$u(x)$ denotes the \emph{fluence} (that is, the laser energy per unit area at a point $x$), $\mu(x)$ the \emph{absorption coefficient} (the photon absorption probability per unit length) and $D(x)=\frac{1}{3(\mu + \mu_s')}$ (where $\mu_s'(x)$ denotes the \emph{reduced scattering coefficient}) the \emph{diffusion coefficient}. Both $\mu$ and $D$ vary spatially and depend on the wavelength of the laser excitation. For details and a derivation of the diffusion approximation, we refer to \cite{Arr99,WanWu07}. 

In the literature, \eqref{eq:diff_eq} is commonly augmented with Dirichlet boundary conditions (which, in practice, might not be known) or, at interfaces with non-scattering media, Robin-type boundary conditions (see, for instance, \cite{WanWu07}).

In this model, the absorbed laser energy $\E(x)$ is given by 
\begin{equation}
\label{eq:def_E}
	\E(x)=\mu(x) u(x).
\end{equation}
The ultrasound initial pressure $\Gamma$ obtained by photoacoustic imaging is proportional to the absorbed energy $\E$, so we have
\begin{equation}
\label{eq:def_H}
	\H(x)=\Gamma(x) \E(x)=\Gamma(x) \mu(x) u(x). 
\end{equation}
The (spatially varying) dimensionless constant $\Gamma$ is called the Gr\"{u}neisen parameter, its value corresponds to the conversion efficiency from change in thermal energy to pressure. 

Hence, the goal in qPAT is to find the parameters $\mu,D,\Gamma$ in a domain $\Omega$ given 
\begin{equation*}
	\H^k = \Gamma \mu u^k, \quad k=1,\ldots,K
\end{equation*} 
where $u^k$ solves \eqref{eq:diff_eq} in $\Omega$ (here and in the following, the index $k$ corresponds to varying laser excitation directions). 

Previous work on this problem (and variations of it) can be found, e.g., in \cite{AmmBosJugKan11,BalRen11a,BalRen12,BalUhl10,BanBagVasRoy08,CoxArrBea09,CoxArrKoeBea06,GaoOshZha12,LauCoxZhaBea10,RenGaoZhao13,SarTarCoxArr13,ShaCoxZem11,TarCoxKaiArr12,YuaZahJia06,Zem10}. For a more comprehensive list, we refer to the review article \cite{CoxLauArrBea12} by Cox et al. 

In particular, Bal and Ren showed (see \cite{BalRen11a}) that unique reconstruction of all three parameters $\mu,D,\Gamma$ is impossible, independent of the number of measurements $\H^k$. They suggested to overcome this problem by the use of \emph{multi-spectral data} (i.e., multiple photoacoustic measurements generated by laser excitations at different wavelengths). Using these data, unique reconstruction of all three material parameters (at the respective wavelengths used), becomes possible \cite{BalRen12}. 

In our paper, we take a different approach and propose a restriction to piecewise constant $\mu,D,\Gamma$. Similar restrictions (due to the large number of publications which use this approach we only provide a small selection of references) have been proposed for \emph{Diffusion Optical Tomography} (e.g., \cite{ArrDorKaiKolSchTarVauZac06,Har09,KolVauKai00,ZacScwKolArr09}) and \emph{Conductivity Imaging} (e.g., \cite{BerFra11,Dru98,KimKwoSeoYoo02,RonSan01}). 

For our problem, it turns out that the reconstruction problem becomes a lot simpler and admits a unique solution for all three parameters $\mu,D,\Gamma$. 

The result is based on an analytical, explicit reconstruction procedure consisting of two steps. First, we recover the regions where $\mu,D,\Gamma$ are constant by finding the discontinuities of photoacoustic data $\H$ and its derivatives up to second order (see Proposition \ref{prop:jump_detection}). In the second step, we determine the actual values of $\mu,D,\Gamma$ from the jumps of $\H$ and $\nabla \H \cdot \nu$ (the normal derivatives) across the obtained region boundaries (cf. Proposition \ref{prop:uniqueness}). Our result holds under certain conditions on the parameters $\mu,D,\Gamma$ and the direction of $\nabla u$. We emphasize that we don't necessarily require that $u|_{\partial\Omega}$ is known (which may not be the case in practice) or that specific boundary conditions hold on $\partial\Omega$. Instead, we use reference values of the parameters for reconstruction, i.e., values of one of the pairs $(\mu(x),\Gamma(x))$ or $(D(x),\Gamma(x))$ at a single point $x \in \Omega$. 

Numerically, the reconstruction method we present heavily relies on an efficient \emph{jump detection} algorithm (using a computational edge detection method) \emph and subsequent \emph{3D-image segmentation}, which provides a connection with image analysis. 

The paper is organized as follows. In section \ref{sec:illposed}, we recap some of the non-uniqueness results for the qPAT problem in literature. In section \ref{sec:uniqueness}, we prove unique solvability for piecewise constant $\mu,D,\Gamma$. In section \ref{sec:numerics}, we give an example of how our ideas can be applied numerically. The last section contains two concrete numerical examples where the reconstruction method is applied to simulated data (with one data set FEM-generated and one data set generated by Monte Carlo simulations). The paper ends with a conclusion.

\section{Ill-posedness of qPAT with smooth parameters}
\label{sec:illposed}

In this section, we review some of the non-uniqueness results for quantitative photoacoustic tomography. For simplicity of presentation, we augment (in this section only) equation \eqref{eq:diff_eq} with Dirichlet boundary conditions, so we have 

\begin{equation}
\label{eq:diff_eq_dirichlet}
 \begin{aligned}
	-\div(D(x) \nabla u(x)) + \mu(x) u(x) &= 0 \quad \text{in $\Omega \subset \mathbb{R}^3$}\\
	u(x)|_{\partial \Omega} &= f(x).
 \end{aligned}	
\end{equation}

The boundary values represent the laser illumination of one particular experiment. In this section, we assume $f$ is known, satisfies $f>0$ and is sufficiently smooth. 

It is well-known and has been shown numerically (see \cite{CoxArrBea09,ShaCoxZem11}) that even when the Gr\"{u}neisen coefficient $\Gamma$ is known (so the absorbed energy $\E=\mu u$ can be calculated from $\H$), different pairs of diffusion and absorption coefficients may lead to the same absorbed energy map $\E$. To see this analytically, for given smooth coefficients $D,\mu >0$ let $u(D,\mu)$ be the corresponding smooth solution of \eqref{eq:diff_eq_dirichlet} and $\E(\mu,D)=\mu u(D,\mu)$ the absorbed energy. By the strong maximum principle (see \cite[Theorem 3.5]{GilTru01}), $u(D,\mu) > 0$ in $\Omega$ (since $f >0$). 

Moreover, for fixed $\E=\E(\mu,D)$, let us denote by $v(\tilde D)$ the solution of 
\begin{equation}
\label{eq:diff_eq_modified}
 \begin{aligned}
	\div(\tilde D(x) \nabla v(x)) &= \E(x) \quad \text{in $\Omega$} \\
	v(x)|_{\partial \Omega} &= f(x).
 \end{aligned}
\end{equation}
Note that $v(D)=u(D,\mu) > 0$. Then, for every $\tilde{D}$ with $\norm{D-\tilde{D}}_{1,\infty} < \epsilon$ (with $\epsilon$ small enough), we also have $v(\tilde{D}) > 0$. 
To see this, note that 
\begin{align*}
\div(\tilde{D}\nabla(v(\tilde{D}) - v(D)))&=-\div((\tilde{D}-D)\nabla v(D)) \quad \text{in $\Omega$} \\
(v(\tilde{D}) - v(D))|_{\partial\Omega} &= 0
\end{align*}
Using a priori bounds \cite[Theorem 3.5]{GilTru01}, 
\begin{equation*}
\norm{v(\tilde{D})-v(D)}_{\infty} \leq C_1 \norm{\div((\tilde{D}-D)\nabla v(D))}_\infty \leq C_2 \norm{\tilde{D} - D}_{1,\infty}, 
\end{equation*}
which implies $v(\tilde D) >0$ if $\epsilon$ is sufficiently small. 

Now, taking $\tilde{\mu}=\frac{\E(\mu,D)}{v(\tilde{D})}$, we get 
\begin{equation*}
	\div(\tilde D \nabla v(\tilde D)) - \tilde \mu v(\tilde D) = 0.
\end{equation*}
Hence, $v(\tilde{D})=u(\tilde{D},\tilde{\mu})$ and $\E(\tilde \mu,\tilde D)= \tilde{\mu} u(\tilde{D},\tilde{\mu}) = \E(\mu,D)$, which shows that infinitely many pairs of 
coefficients may create the same absorbed energy map.

This nonuniqueness can be overcome by varying $f$ (i.e., changing the illumination pattern), obtaining multiple absorbed energy maps. This approach is called \emph{multi-source} quantitative photoacoustic tomography. Bal and Ren \cite{BalRen11a} showed that while this additional information leads to unique reconstruction of $\mu,D$ from $\E$, finding three unknown parameters $\mu,D,\Gamma$ given $\H=\Gamma \mu u$ is still impossible, independent of the number of illuminations (any more than two do not add any information). In fact, they showed that for \emph{any} given Lipschitz continuous $\mu$, $D$ or $\Gamma$ the other two parameters can be chosen such that given initial pressures $\H^k=\Gamma \mu u^k$ (for multiple illumination patterns $f^k$) are generated.

\begin{example}
\label{ex:nonuniqueness}
Given any set of parameters $(\mu,D,\Gamma)$, for every $\lambda > 0$, $(\lambda \mu,\lambda D,\frac{1}{\lambda} \Gamma)$ generate the same measurements, since  \eqref{eq:diff_eq_dirichlet} is invariant under simultaneous scaling of $\mu$ and $D$. 
\end{example}

This simple example shows that even for constant parameters knowledge of $f$ and $\H$ is insufficient to determine $\mu,D,\Gamma$. Hence, more prior information about the unknown parameters will be necessary in order to get a unique solution. 

\section{Reconstruction of piecewise constant parameters}
\label{sec:uniqueness}

To overcome this essential non-uniqueness, we assume that $\mu,D,\Gamma$ are piecewise constants. That is, for some partition $(\Omega_m)_{m=1}^M$ of $\Omega \subset \mathbb{R}^3$,  

\begin{equation}
\label{eq:piecewise_const}
	\overline\Omega = \bigcup_{m=1}^M \overline\Omega_m, \enskip \mu = \sum_{m=1}^M \mu_m 1_{\Omega_m}, \enskip D = \sum_{m=1}^M D_m 1_{\Omega_m}, \enskip \Gamma = \sum_{m=1}^M \Gamma_m 1_{\Omega_m}.
\end{equation}

Since the parameters are discontinuous, we need a generalized solution concept. Under certain additional conditions (which we explain in detail in Appendix \ref{sec:transmission_cond}) a weak solution $u$ of \eqref{eq:diff_eq_dirichlet} with piecewise constant parameters $\mu, D$ can be characterized by 

\begin{equation}
\label{eq:u_continuity}
	u \in C^\alpha(\overline{\Omega})
\end{equation}

for some $\alpha >0$ and, for $m=1,\ldots,M$,

\begin{equation}
\label{eq:diff_eq_scalar}    
\begin{aligned}
	&u_m := u|_{\Omega_m} \in C^\infty(\Omega_m) \\
	&D_m \Delta u_m - \mu_m u_m = 0 \quad \text{in } \Omega_m
\end{aligned}
\end{equation}

and, almost everywhere on interfaces $I_{mn}:=\partial\Omega_m \cap \partial\Omega_n$, 

\begin{equation}
\label{eq:transmission_cond}
	D_m \nabla u_m \cdot \nu = D_n \nabla u_n \cdot \nu \quad  \text{(for any normal vector $\nu$).} 
\end{equation}

The transmission condition \eqref{eq:transmission_cond} is ill-defined on corners and intersections of multiple subregions, therefore we can only expect it to hold almost everywhere.  For details and a derivation, see Appendix \ref{sec:transmission_cond}. The transmission condition \eqref{eq:transmission_cond} can also be derived physically (rather than starting from a weak solution), it is accurate within the scope of the diffusion approximation  \cite{Aro95,RipNie99}.

From now on, we consider $u_m$ and $\H_m:=\H|_{\Omega_m}=\Gamma_m \mu_m u_m$ (and their derivatives up to second order) continuously extended (from the inside) to $\partial\Omega_m$. We emphasize 
that for $\nabla u_m$ and $\nabla \H_m$, this may only be possible for almost all points (with respect to the surface measure), see Appendix \ref{sec:transmission_cond}. 

We also assume that $u$ is strictly positive and bounded from above in $\overline\Omega$. 

In the following Proposition \ref{prop:jump_detection}, we show that the jump set $\bigcup_m \partial\Omega_m$ of piecewise constant parameters $\mu,D,\Gamma$ can be determined from  photoacoustic initial pressure data $\H=\Gamma \mu u$. For $k \geq 0$, denote by 
\begin{equation*}
J_k(f)=\Omega \setminus \bigcup \{ B \subset \Omega \big| B \text{ is open and } f \in C^k(B) \}
\end{equation*}
the set of discontinuities of a function $f \in L^\infty(\Omega)$ and its derivatives up to $k$-th order.

We require an assumption on $\nabla u$ and the unknown parameters $\mu,D,\Gamma$. For all $x \in J_0(D) \setminus (J_0(\Gamma \mu) \cup J_0(\frac{\mu}{D})) \subset \partial\Omega_m \cap \partial\Omega_n$  (i.e., interfaces of $D$ which are not interfaces of $\Gamma \mu$ and $\frac{\mu}{D}$) we require that the fluence $u$ satisfies almost everywhere (where $\nu$ denotes a normal vector on $\partial\Omega_m \cap \partial\Omega_n$),
\begin{equation}
\label{eq:cd_assumption_1}
	|\nabla u_n(x) \cdot \nu(x)| >  0 \quad (\stackrel{\eqref{eq:transmission_cond}}{\iff} |\nabla u_m(x) \cdot \nu(x)| > 0).
\end{equation}

\begin{proposition}
\label{prop:jump_detection}

Let $\mu,D,\Gamma$ be of the form \eqref{eq:piecewise_const} and $u=\sum_m u_m 1_{\Omega_m}$ and $\H=\sum_m \H_m 1_{\Omega_m}$ 
the corresponding fluence and initial pressure distributions satisfying condition \eqref{eq:cd_assumption_1} in $\Omega$. Then, 

\begin{equation*}
	\overline{J_0(\mu)} \cup \overline{J_0(D)} \cup \overline{J_0(\Gamma)} = \overline{J_2(\H)}.
\end{equation*}

\end{proposition}

\begin{proof}

Let $B \subset \Omega$ be an open ball with $B \cap (J_0(\mu) \cup J_0(D) \cup J_0(\Gamma) ) = \emptyset$. Since $u$ solves an elliptic PDE with constant coefficients in $B$, we have $u \in C^\infty(B)$ by interior regularity. Hence $\H \in C^\infty(B)$ (since $\Gamma\mu$ is constant in $B$), which implies $J_2(\H) \subset J_0(\mu) \cup J_0(D) \cup J_0(\Gamma)$. 

To show the converse, take $x \in \Omega$ such that $x \in J_0(\mu) \cup J_0(D) \cup J_0(\Gamma)$ (that is, one of the parameters jumps at $x$). We have to show that $x \in \overline{J_2(\H)}$. 

Let $m,n$ be such that $x \in I_{mn}=\partial \Omega_m \cap \partial \Omega_n$. We distinguish three cases:

\begin{enumerate}

	\item[(1)] $\Gamma_m \mu_m \neq \Gamma_n \mu_n$: Since $u$ is continuous across $I_{mn}$ (cf. Appendix \ref{sec:transmission_cond}), $\H=\Gamma\mu u$ is discontinuous at $x$, 
	so we get $x \in J_0(\H) \subset J_2(\H)$.
	
	\item[(2)] $\Gamma_m \mu_m = \Gamma_n \mu_n, \frac{\mu_m}{D_m} \neq \frac{\mu_n}{D_n}$: From \eqref{eq:diff_eq_scalar} and $u \in C(\overline\Omega)$ we get
	\begin{equation*}
	\Delta u_m(x) = \frac{\mu_m}{D_m} u_m(x) \neq \frac{\mu_n}{D_n} u_n(x) = \Delta u_n(x).
	\end{equation*}
	
	Hence $x \in J_2(u)$, which implies $x \in J_2(\H)$ since $\Gamma \mu$ is constant in $\Omega_m \cup \Omega_n$.  

	\item[(3)] $\Gamma_m \mu_m = \Gamma_n \mu_n, D_n \neq D_m$: First, let  $x \in I_{mn}$ be a point where the transmission condition \eqref{eq:transmission_cond} and \eqref{eq:cd_assumption_1} hold (by assumption, this is the case for almost all points with respect to the surface measure). We have 
\begin{equation*}
\begin{aligned}
\left|\nabla u_m(x) - \nabla u_n(x) \right| &\geq \left| (\nabla u_m(x)- \nabla u_n(x))\cdot \nu(x) \right|  \\
&\geq \left|1-\frac{D_m}{D_n} \right| \left|\nabla u_m(x) \cdot \nu(x) \right| > 0.
\end{aligned}
\end{equation*}
This shows that $x \in J_1(u)$, which implies $x \in J_1(\H)$ and thus $x \in J_2(\H)$. By taking the closure, we get $x \in \overline{J_2(\H)}$ for all $x \in I_{mn}$.
\end{enumerate}
The cases (1)-(3) cover all possibilities, since otherwise all three parameters $\mu,D,\Gamma$ would be constant in $\Omega_m \cup \Omega_n$. 
\end{proof}

Proposition \ref{prop:jump_detection} shows that we can obtain the parameter discontinuities (in regions where \eqref{eq:cd_assumption_1} holds) via the set $J_2(\H)$. In fact, the proof tells us that $\H$, $\nabla \H$ or $\Delta \H$ have jumps at discontinuities of $\mu$, $D$ or $\Gamma$. That is, images of the gradient and Laplacian of the data $\H$ show material inhomogeneities not visible in $\H$.

In the next Proposition, we show how to recover piecewise constant parameters $\mu,D,\Gamma$ once their jump set $\bigcup_m \partial\Omega_m$ is known (e.g., from Proposition \ref{prop:jump_detection}). 
Knowledge of boundary values of $u$ alone is insufficient to fully determine the parameters (see Example \ref{ex:nonuniqueness}). We also have to require knowledge of the parameters in some 
$\Omega_n \subset \Omega, \ n \in \{1,\ldots,M\}$. Using the continuity of $u$, \eqref{eq:diff_eq_scalar} and \eqref{eq:transmission_cond}, we will show that these reference values combined with photoacoustic measurements $\H=\Gamma \mu u$ suffice to determine $\mu,D,\Gamma$ everywhere.

For this result, we again need an assumption on $\nabla u$. For every interface $I_{mn}=\partial\Omega_m \cap \partial\Omega_n$ with normal vector $\nu(x)$, we require the existence of some $x \in I_{mn}$ with
\begin{equation}
\label{eq:cd_assumption_2}
\nabla u_n(x) \cdot \nu(x) \neq 0 \quad (\stackrel{\eqref{eq:transmission_cond}}{\iff}  \nabla u_m(x) \cdot \nu(x) \neq 0),
\end{equation}
that is, on every interface $I_{mn}$ there must exist a point where $\nabla u$ is not tangential.

\begin{proposition}
\label{prop:uniqueness}
Let $\mu,D,\Gamma$ be of the form \eqref{eq:piecewise_const} (with the decomposition $(\Omega_m)$ of $\Omega$ known). Furthermore, let $u$ and $\H=\Gamma \mu u$ be corresponding fluence and initial pressure distribution which satisfy condition \eqref{eq:cd_assumption_2} on every interface $I_{mn} \subset \Omega$.  Furthermore, let $(\mu_n,D_n,\Gamma_n)$ be known for some $n$. Then the parameters $\mu,D,\Gamma$ can be determined uniquely from $\H$.
\end{proposition}
\begin{proof}
Let $\Omega_m$ be a neighbouring subregion to $\Omega_n$ and denote by $I_{mn}=\partial\Omega_m \cap \partial\Omega_n$ the interface. By continuity of $u$ and \eqref{eq:def_H}, we have for all $y \in I_{mn}$

\begin{equation}
\label{eq:unique_1}
	\Gamma_m \mu_m = \frac{\H_m(y)}{\H_n(y)} \ \Gamma_n \mu_n 
\end{equation}

so from the reference values and $\H$ we can calculate $\Gamma\mu$ on neighbouring $\Omega_m$. 

Next, let $x \in \partial\Omega_m \cap \partial\Omega_n$ such that $\nabla u_n(x) \cdot \nu(x) \neq 0$. Using \eqref{eq:transmission_cond} and $\nabla \H_k = \Gamma_k \mu_k \nabla u_k$ for all $k$ (since the parameters are constant in $\Omega_k$) we get

\begin{equation}
\label{eq:unique_2}
	\frac{D_m}{\Gamma_m \mu_m} = \frac{(\nabla \H_n \cdot \nu)(x)}{(\nabla \H_m \cdot \nu)(x)} \frac{D_n}{\Gamma_n \mu_n}.
\end{equation}  

Finally we get for all in $z \in \Omega_m$, from \eqref{eq:diff_eq_scalar} and $\Delta H_m = \Gamma_m \mu_m \Delta u_m$ in $\Omega_m$,

\begin{equation}
\label{eq:unique_3}
\frac{\mu_m}{D_m}=\frac{\Delta \H_m(z)}{\H_m(z)}.
\end{equation}

The equations \eqref{eq:unique_1}-\eqref{eq:unique_3} suffice to obtain $\mu_m$, $D_m$ and $\Gamma_m$, since we have 
\begin{equation}
\label{eq:unique_final}
	(\mu,D,\Gamma)=\left( A B C , A B, \frac{1}{B C}   \right), 
\end{equation}
for $A=\Gamma\mu,\ B=\frac{D}{\Gamma\mu},\ C=\frac{\mu}{D}$.

By iterating over all interfaces, we can find $\mu,D,\Gamma$ everywhere in $\Omega$.
\end{proof}

Note that in Proposition \ref{prop:uniqueness}, no knowledge of boundary values of $u$ is required, values of the parameters $\mu_n,D_n,\Gamma_n$ in some $\Omega_n$ are enough. In fact, knowledge of two of the three parameters already suffices, as we will show in the following Proposition.

\begin{proposition}
\label{prop:local_uniqueness}
For a given $n$, the constants $(\mu_n,D_n,\Gamma_n)$ can be determined uniquely from photoacoustic data $\H_n=\Gamma_n \mu_n u_n$ and knowledge of one of the pairs $(\mu_n,\Gamma_n)$ or $(D_n,\Gamma_n)$. If $u(x)$ is known for some $x \in \Omega_n$, knowing one of the three constants is enough. If only one of the parameters $\mu_n,D_n,\Gamma_n$, only $u_n$, or the only pair $(\mu_n, D_n)$ is known, $(\mu_n,D_n,\Gamma_n)$ cannot be determined uniquely.
\end{proposition}

\begin{proof}
From \eqref{eq:diff_eq_scalar} and $\H=\Gamma\mu u$, we know that in $\Omega_n$

\begin{equation*}
 \begin{aligned}
	D_n \Delta u_n - \mu_n u_n &= 0 \\
	\Gamma_n \mu_n u_n &= \H_n,
 \end{aligned}
\end{equation*}

which is equivalent to

\begin{equation}
\label{eq:diff_eq_system}
 \begin{aligned}
	\Gamma_n D_n \Delta u_n &= \H_n \\
	\Gamma_n \mu_n u_n &= \H_n  \\
	\Gamma_n \mu_n \Delta u_n &= \Delta \H_n.
 \end{aligned}
\end{equation}

Here, one can immediately see that if $u(x)$ is known for some $x \in \Omega_n$, we can calculate $u_n(y)=\frac{1}{\Gamma_n \mu_n}\H_n(y) =\frac{u(x)}{\H(x)} \H_n(y)$ for all $y \in \Omega_n$ and thus also $\Delta u_n$ . Clearly, $(\mu_n,D_n,\Gamma_n)$ can now be determined from \eqref{eq:diff_eq_system} if one of the parameters is known.

Likewise, given one of the pairs $(\mu_n,\Gamma_n)$ or $(D_n,\Gamma_n)$ we can to calculate all three constants $(\mu_n,D_n,\Gamma_n)$. 

Knowledge of $(\mu_n,D_n)$, on the other hand, is insufficient because $\lambda u_n$, $\frac 1 \lambda \Gamma_n$ satisfy \eqref{eq:diff_eq_system} for given $(\mu_n,D_n)$ for all $\lambda > 0$. 
Similarly, the system is underdetermined if only $u_n$ or $\Gamma_n$ is known.

\end{proof}

Conditions \eqref{eq:cd_assumption_1} and \eqref{eq:cd_assumption_2} are vital for unique reconstruction. For instance, using Lemma \ref{prop:transmission_cond_converse} one can see that, $u(x,y,z)=e^x$ is a weak solution of \eqref{eq:diff_eq} in $\mathbb{R}^3$ for both
\begin{equation*}
\mu \equiv 1, \quad D \equiv 1, \quad \Gamma \equiv 1
\end{equation*}
and
\begin{equation*}
\tilde \mu = 
\left\{ 
	\begin{array}{ll}
		1  & \text{ if } y \geq 0 \\
	    \lambda & \text{ if } y < 0
	\end{array}
\right. \hspace{-0.6 em}, \quad
\tilde D = 
\left\{ 
	\begin{array}{ll}
		1  & \text{ if } y \geq 0 \\
		\lambda  & \text{ if } y < 0
	\end{array}
\right. \hspace{-0.6 em}, \quad
\tilde \Gamma = 
\left\{ 
	\begin{array}{ll}
		1  & \text{ if } y \geq 0 \\
		\frac 1 \lambda  & \text{ if } y < 0
	\end{array}
\right. \hspace{-0.6 em}, \quad \lambda > 0
\end{equation*}
since $u$ is a classical solution on both sides of the interface $\{y=0\}$ and it satisfies $\nabla u \cdot \nu=0$. Furthermore, both parameter sets generate the same data $\H(x,y,z)=e^x$.

More generally, parts of interfaces where condition \eqref{eq:cd_assumption_1} fails to hold don't necessarily lie in $J_2(\H)$ and may thus be invisible to our reconstruction procedure (depending on the geometry, this may also lead to follow-up errors). If condition \eqref{eq:cd_assumption_2} fails to hold, it might not be possible to determine $\mu,D,\Gamma$ everywhere. 

To overcome this problem, we can use additional measurements (with different illumination directions) and hope that the location of critical points and gradient directions change. In particular, if photoacoustic data $(\H^k)_{k=1}^K$ corresponding to solutions $(u^k)_{k=1}^K$ of \eqref{eq:diff_eq} that satisfy for almost all $x \in \Omega$ 
\begin{equation}
\label{eq:detcond}
		\max_{i,j,k} \left| \det(\nabla u^i(x),\nabla u^j(x),\nabla u^k(x)) \right| > 0 \\
\end{equation}
are available, on every $x \in I_{mn}$, one of the measurements satisfies \eqref{eq:cd_assumption_1} (since $\nabla u^i(x),\nabla u^j(x),\nabla u^k(x)$ form a basis). 
With a similar argument as in Proposition \ref{prop:jump_detection} one can show that in this case
\begin{equation*}
J_0(\mu) \cup J_0(D) \cup J_0(\Gamma) = \bigcup_{k=1}^K J_2(\H^k),
\end{equation*}
so unique reconstruction of $\mu,D,\Gamma$ in $\Omega$ can be guaranteed. To our knowledge, no method to force condition \eqref{eq:detcond} by boundary conditions or choice of source is known, however, its validity can be checked by looking at the data $(\H^k)_{k=1}^K$.

\section{Numerical reconstruction}
\label{sec:numerics}

In this section, we show how the results in the last section can be utilized numerically. Our goal is to estimate unknown piecewise constant parameters $\mu,D,\Gamma$ from noisy three-dimensional photoacoustic data $(\H^k)_{k=1}^K$ (with varying boundary excitations) sampled on a regular grid. 
 
We propose a two-step reconstruction:
\begin{itemize}

	\item[(1)] Detect jumps in $(\H^k)_{k=1}^K$, $(\nabla \H^k)_{k=1}^K$, $(\Delta \H^k)_{k=1}^K$ and use the obtained surfaces to segment the image domain $\Omega$ to estimate subregions $(\hat{\Omega}_m)_{m=1}^M$ where the parameters are constant (and thus $\H$ is smooth).
	\item[(2)] Given $(\hat\Omega_m)_{m=1}^M$ and reference values $(\mu_1,D_1,\Gamma_1)$, use the jump values of $(\H^k)_{k=1}^K$ and $(\nabla \H^k \cdot \nu)_{k=1}^K$  across the estimated interfaces and values of $(\frac{\Delta \H^k}{\H^k})_{k=1}^K$ to find $\mu,D,\Gamma$ everywhere in $\Omega$ (using equations \eqref{eq:unique_1}-\eqref{eq:unique_final}).

\end{itemize}

\subsection{Finding regions where the parameters are constant}
\label{sub:num_regions}

In the proof of Proposition \ref{prop:jump_detection}, one can see that in regions $\Omega$ where \eqref{eq:cd_assumption_1} holds, discontinuities of the (piecewise constant) parameters $\mu,D,\Gamma$ correspond to jumps of $\H$, $\nabla \H,\Delta \H$. We want to use \emph{computational edge detection} to find these jumps. 

We start by finding jumps in $\H$. Since they are multiplicative (i.e., $\frac{H_m}{H_n}$ is constant on $I_{mn}$), we apply a logarithm transformation to get constant jumps along the interfaces. In fact, let $I_{mn}=\partial\Omega_m \cap \partial\Omega_n$ be an interface of the parameters $\mu,D,\Gamma$. Since $u_m=u_n$ on $I_{mn}$, we have 

\begin{equation*}
\label{eq:jump_Gamma_mu}
	\left|\log(\H_n)-\log(\H_m) \right|=\left|\log(\Gamma_n \mu_n)-\log(\Gamma_m\mu_m) \right|  \text{ on } I_{mn},
\end{equation*}
so jumps in $\Gamma \mu$ lead to jumps of equal magnitude in $\log \H$. 

Next, we show that jumps in $D$ (that are large enough compared to those in $\Gamma\mu$) lead to jumps in $\log|\nabla \H|$. We restrict our search domain to $\Omega' \subset \Omega$ such that $|\nabla \H| \geq d > 0$ holds in $\Omega'$. Due to continuity of $u$ we have $\nabla u_m \cdot \tau = \nabla u_n \cdot \tau$ (for tangential vectors $\tau$) on parts of $I_{mn}$ that are $C^1$. Thus, we obtain on parts of $I_{mn} \cap \Omega'$ where \eqref{eq:transmission_cond} holds,
\begin{equation*}
\label{eq:jump_nabla_u}
\begin{aligned}
	\frac{|\nabla u_n|^2}{|\nabla u_m|^2} &= \frac{(\nabla u_n \cdot \nu)^2 + (\nabla u_n \cdot \tau)^2}{|\nabla u_m|^2} \stackrel{\eqref{eq:transmission_cond}}{=} \left(\frac{D_m}{D_n}\right)^2 \frac{ (\nabla u_m \cdot \nu)^2}{|\nabla u_m|^2} + \frac{ (\nabla u_m \cdot \tau)^2}{|\nabla u_m|^2} \\
		&= 1 + \left( \left(\frac{D_m}{D_n}\right)^2 - 1 \right) \cos(\alpha_m)^2,  \\
\end{aligned}
\end{equation*}
where $\alpha_m$ denotes the angle between the unit normal $\nu$ and $\nabla u_m$. Using $D_m \geq D_n$ (without loss of generality, otherwise we swap indices),  we get 
\begin{align*}
	\left| \log|\nabla u_n| - \log|\nabla u_m| \right| &\geq \frac 1 2 \log \left( 1 + \left(  e^{2 |\log D_m - \log D_n|} - 1 \right) \min_{\Omega,k} \cos(\alpha_k)^2 \right) \\
	&:= \gamma\left(| \log D_m - \log D_n| \right).
\end{align*}
If $\min_{\Omega,k} \cos(\alpha_k)^2 > 0$ holds in $\Omega'$, the function $\gamma$ is positive, strictly increasing and unbounded. Hence, using the reverse triangle inequality,
\begin{equation*}
\label{eq:jump_D}
\begin{aligned}
	|\log(|&\nabla \H_n|) - \log(|\nabla \H_m|) | = \left|\log(\Gamma_n\mu_n)-\log(\Gamma_m\mu_m)+\log|\nabla u_n|-\log|\nabla u_m| \right| \\
	&\geq \left|\log|\nabla u_n|-\log|\nabla u_m| \right| - \left|\log(\Gamma_n\mu_n)-\log(\Gamma_m\mu_m) \right| \\
	&\geq \gamma \left(| \log D_m - \log D_n| \right) - |\log(\Gamma_n \mu_n) - \log(\Gamma_m \mu_m)|.
\end{aligned}
\end{equation*}

Finally, since $\frac{\Delta \H_m}{\H_m}=\frac{\mu_m}{D_m}$ for all $m$, we get on $I_{mn}$
\begin{equation*}
\left|\log \left(\frac{\Delta \H_m}{\H_m} \right)-\log \left(\frac{\Delta \H_n}{\H_n} \right)\right| = \left|\log \left(\frac{\mu_m}{D_m} \right) - \log \left(\frac{\mu_n}{D_n} \right)\right|
\end{equation*}
which shows that jumps in $\log \left(\frac{\mu}{D}\right)$ lead to jumps of equal magnitude in $\log \left(\frac{\Delta \H}{\H}\right)$.

To ensure that $|\nabla \H| \geq d > 0$ holds, we enforce a minimum for $|\nabla \H|$ (to avoid creating singularities). We counter failure of $\min_{\Omega,k} \cos(\alpha_k)^2 > 0$ by using additional measurements.

To estimate $(\Omega_m)_{m=1}^M$ given noisy data $\H^k$, we first look for jumps in $\log(\H^k)$, then $\log |\nabla \H^k|$ and last in $|\log(\Delta \H^k) - \log(\H^k)|$. More precisely, we proceed as follows: 

\begin{itemize}

	\item[(1)] Find $\hat{J}_0 \subset \Omega$, a surface across which $\log\H^k$ jumps more than some threshold $\tau_0$. Segment the domain $\Omega$ using $\hat{J}_0$ (i.e., find the connected components of $\Omega \setminus \hat{J}_0$), giving subsets $(\hat{\Omega}^0_i)_{i=1}^I$, an estimate of the regions where $\Gamma\mu$ is constant.
	\item[(2)] In all $\hat{\Omega}^0_i$, search for jumps in $\log |\nabla \H^k|$ that are bigger than threshold $\tau_1$, obtaining sets $\hat{J}^i_1 \subset \hat{\Omega}^0_i$. Take $\hat{J}_1=\bigcup_i \hat{J}^i_1 \cup \hat{J}_0$ and segment $\Omega$ using $\hat{J}_1$ to get $(\hat{\Omega}^1_i)_{i=1}^I$, an estimate for the regions where $\Gamma\mu$ and $D$ are constant.
	\item[(3)] In all $\hat{\Omega}^1_i$, search for jump sets $\hat{J}^n_2 \subset \hat{\Omega}^1_n$ of $|\log \Delta \H^k - \log \H^k|$, with values above lower threshold $\tau_2$. We get $\hat{J}_2=\bigcup_n \hat{J}^i_2 \cup \hat{J}_1$, our estimate for $J(\Gamma\mu)$ Finally, by segmenting $\Omega$ using $\hat{J}_2$ we get $(\hat{\Omega}_m)_{m=1}^M$, an estimate for the regions where $\Gamma\mu,D,\frac \mu D$ (and thus also parameters $\mu,D,\Gamma$) are constant.

\end{itemize}

We can take advantage of multiple measurements $(\H^k)_{k=1}^K$ (with different illuminations) by detecting edges separately for all $\H^k$ (and their derivatives) and joining the edge sets prior to segmentation in each step (1)-(3), or simpler, by averaging the input data for edge detection in each steps (1)-(3) (we implemented the second strategy). Using multiple measurements can be vital to counter locally missing contrast due to failure of condition \eqref{eq:cd_assumption_1} or close to extremal points of $\H$. 

For the actual jump detection, we use \emph{Canny edge detection} in differential form as proposed by Lindeberg (cf. \cite{Can86} and \cite{Lin98}). See Appendix \ref{sec:canny} for a short description of the method.

\subsection{Estimating optical parameters}
\label{sub:est_parameters}

In the second stage of the reconstruction process, we want to estimate $\mu,D,\Gamma$ from photoacoustic data $(\H^k)_{k=1}^K$  (sampled on a regular grid) given an estimate of the sets $(\Omega_m)_{m=1}^M$ (from the previous section) and reference values, for which we choose $(\mu_1,D_1,\Gamma_1)$ (without loss of generality). For simplicity, we first explain the procedure for a single measurement $\H$.

In the proof of Proposition \ref{prop:uniqueness}, evaluations of $\H_m$,$\nabla \H_m \cdot \nu$ and $\Delta \H_m$ at isolated points were sufficient to obtain all parameters. In the presence of noise and discretization error it is, however, better to use all the jump information available. Rather than calculating $\mu_m,D_m,\Gamma_m$ in an arbitrary order using equations \eqref{eq:unique_1}-\eqref{eq:unique_final} we use a least-squares fitting method to calculate $\Gamma\mu,\frac{D}{\Gamma\mu},\frac{\mu}{D}$ in all $\Omega_m$ simultaneously. 

Since the data $\H$ contains noise and is only known on a grid, we can only calculate the values of $\H_m$ (whose values may not be known precisely on interfaces), $\nabla \H_m \cdot \nu$ and $\Delta \H_m$ up to some error. For $m=1,\ldots,M$, $y \in \partial\Omega_m$,  $z \in \Omega_m$ and $x \in \partial\Omega_m$ with $|\nabla\H_m(x) \cdot \nu(x)| > 0$ let $h_m, g_m, l_m$ be the approximations 
\begin{equation}
\label{eq:estpar_approx}
\begin{aligned}
h_m(y) &\approx \log \H_m(y) \\
g_m(x) &\approx \log |\nabla\H_m(x) \cdot \nu(x)| \\
l_m(z) &\approx \log \left( \frac{\Delta \H_m(z)}{\H_m(z)} \right). 
\end{aligned}
\end{equation}

From \eqref{eq:unique_1} and \eqref{eq:unique_2}, we get on $I_{mn}$ 
\begin{equation}
\label{eq:estpar_error_1}
  \log(\Gamma_m \mu_m)-\log(\Gamma_n\mu_n)=\log(\H_m)-\log(\H_n) = h_m - h_n + \epsilon_1 \\
\end{equation}
and 
\begin{equation}
\label{eq:estpar_error_2}
\begin{aligned}
  \log \left(\frac{D_m}{\Gamma_m \mu_m} \right) - \log \left(\frac{D_n}{\Gamma_n \mu_n} \right)  &=  \log|\nabla \H_n \cdot \nu| - \log|\nabla \H_m \cdot \nu|\\
  &=  g_n - g_m + \epsilon_2,
\end{aligned}
\end{equation}
with $\epsilon_1, \epsilon_2$ denoting error terms.  Now, we can estimate 
\begin{equation*}
	\hat a_m \approx \log(\Gamma_m \mu_m), \quad \hat b_m \approx \log \left(\frac{D_m}{\Gamma_m \mu_m} \right)
\end{equation*}
for $m>2$ ($a_1, b_1$ can be calculated from the reference values) by choosing values which minimize the $L^2$-norm of the error terms $\epsilon_1$ and $\epsilon_2$ over all interfaces, that is, by solving the least squares problems 
\begin{equation}
\label{eq:estpar_lsq}
\begin{aligned}
(\hat{a}_2,\ldots,\hat{a}_M) &= \arg\min_{a_2,\ldots,a_M} \sum_{\substack{n,k=1 \\ k>n}}^M \norm{a_k - a_n + h_n - h_k}^2_{L^2(I_{nk})} \\
(\hat{b}_2,\ldots,\hat{b}_M) &= \arg\min_{a_2,\ldots,a_M} \sum_{\substack{n,k=1 \\ k>n}}^M \norm{b_k - b_n - g_n + g_k }^2_{L^2(\tilde I_{nk})},
\end{aligned}
\end{equation}
In the second least squares problem, we restrict the calculation to $\tilde I_{nk}$, a subset of $I_{nk}$ where $g_n,g_k$ are below some bound (i.e., where $|\H_n \cdot \nu|$ and $|\H_k \cdot \nu|$ are not zero).

A simple calculation shows that the optimizers $\hat{a},\hat{b}$ satisfy for $2 \leq m \leq M$
\begin{equation}
\label{eq:estpar_loglin_sol}
\begin{aligned}
	\sum_{k \neq m} (\hat{a}_m - \hat{a}_k) \mathcal{A}(I_{mk}) &= \sum_{k \neq n} \int_{I_{mk}} h_m - h_k \,dS\\
    \sum_{k \neq m} (\hat{b}_m - \hat{b}_k) \mathcal{A}(\tilde I_{mk}) &= \sum_{k \neq n} \int_{\tilde  I_{mk}} g_k - g_m \,dS,
\end{aligned}	
\end{equation}
where $\mathcal{A}(I_{nk})$ denotes the area of the interface $\partial\Omega_m \cap \partial\Omega_n$. Since the corresponding system matrices are irreducibly diagonally dominant, the optimizers $\hat{a},\hat{b}$ are unique (see, e.g., \cite[Theorem 6.2.27]{HorJoh90} ). 

In \eqref{eq:estpar_loglin_sol}, one can see that in the special case where $\Omega_1$ is the background and $\Omega_2,\ldots,\Omega_M$ are inclusions with no shared boundaries, our approach is equivalent to adding to $\log(\Gamma_1 \mu_1)$ (respectively $\log (\frac{D_1}{\Gamma_1 \mu_1}$)) the estimated jump values $h_k - h_1$ (respectively $g_1-g_k$) averaged over $\partial\Omega_k$.

Now, using \eqref{eq:unique_2} and \eqref{eq:estpar_approx}, we have in $\Omega_m, \ m=1,\ldots,M$
\begin{equation}
\label{eq:estpar_error_3}
  \log \left( \frac{\mu_m}{D_m} \right) = \log \left( \frac{\Delta \H_m}{\H_m} \right) = l_m + \epsilon_3
\end{equation}
for some error term $\epsilon$. As before, we can estimate
\begin{equation*}
\hat{c}=\log \left( \frac{\mu}{D} \right )
\end{equation*}
by minimizing the error term $\epsilon_3$. We obtain for $m=1,\ldots,M$
\begin{equation}
\label{eq:estpar_quad_sol}
	\hat{c}_m = \arg\min_c \norm{c - l_m }_{L^2(\Omega_m)}^2 = \frac{1}{\mathcal{V}(\Omega_m)}\int_{\Omega_m} l_m \,dx,
\end{equation}
where $\mathcal{V}(\Omega_m)$ denotes the volume of $\Omega_m$, i.e., we take the mean of $l_m$ in $\Omega_m$. Finally, from $\hat a, \hat b, \hat c$, we can calculate $\hat\mu,\hat D,\hat\Gamma$ with \eqref{eq:unique_final}.

The use of multiple measurements simply amounts to an additional summation in \eqref{eq:estpar_loglin_sol} and \eqref{eq:estpar_quad_sol}, which corresponds to minimizing over the sum of all measurements. 

\subsection{Implementation}
\label{sub:implementation}

We implemented the ideas presented in the last sections in \emph{MATLAB}. The, possibly noisy, photoacoustic pressure data $(\H^k)_{k=1}^K$ is given sampled on a regular 3D-grid with sufficiently high resolution.

Following the scheme presented in \ref{sub:num_regions}, we first estimate subregions $(\hat{\Omega}_m)_{m=1}^M$ where $\mu,D,\Gamma$ are constant by using computational edge detection and then segmenting $\Omega$ using the obtained jump sets.

To detect jumps we use differential \emph{Canny edge detection} (see Appendix \ref{sec:canny} for details). The derivatives are estimated via finite differences (after low-pass filtering with a Gaussian kernel). We obtain jump surfaces with sub-voxel resolution in the form of a triangular mesh. For segmentation, we applied the \emph{MATLAB} image processing toolbox function \emph{bwconncomp}, which works on a voxel level (small holes in the jump sets, for instance at corners, can be closed up by increasing the thickness of the voxelized surfaces). 

Given the jump surfaces and estimated regions $(\hat{\Omega}_m)_{m=1}^M$, in order to approximate $h_m \approx \log\H_m$ and $g_m \approx |\nabla \H_m \cdot \nu|$ (cf. \eqref{eq:estpar_approx}), we fit for every triangular element $e$ (with incenter $y$) of the surface a log-linear function $f_m^e$ to the data $\H_m$ at nearby grid points (using a Gaussian weight function that gives  grid points closer to $y$ a larger weight). By taking $h_m^e = \log f_m^e(y)$ and $g_m^e = \log |\nabla f_m^e(y) \cdot \nu(y)|$ at $y$, we get approximations $h_m$ and $g_m$ that are piecewise constant on the surface elements $e$. We obtain $\hat{a} \approx \log(\Gamma \mu)$ and $\hat{b} \approx \log \left(\frac{D}{\Gamma \mu} \right)$ by solving \eqref{eq:estpar_loglin_sol}.

Similarly, we use \eqref{eq:estpar_quad_sol} to estimate $\hat{c} \approx \log \left( \frac{\mu}{D} \right)$. Here, we locally (at grid points $z$ inside the estimated regions $\hat\Omega_m$) fit quadratic functions $q_m^z$ to the data $H_m$, calculate $l_m = \log \left|  \frac{\Delta q_m^z}{\H_m(z)} \right| \approx \log \left( \frac{\Delta \H_m}{\H_m} \right)$ and average over $z$ to obtain $\hat{c}$ (since the fitting procedure is computationally intensive, this calculation is only performed on a random sample of the grid points, replacing the total average with the sample average).

\section{Numerical examples}
\label{sec:examples}

In this section, we apply the numerical method described in the last section to simulated data. We start with a simple example using FEM-generated data with no added noise. 

In the second example, we work with Monte Carlo generated data with added noise. The Monte Carlo method for photon transfer in random media (which is physically more accurate than the diffusion approximation) converges to solutions of the \emph{radiative transfer equation} and thus satisfies our model \eqref{eq:diff_eq} only approximately (see, e.g., \cite{WanWu07} for details).

\subsection{Example using FEM-generated data}
\label{sub:fem_data}

In the first example, we simulated a single photoacoustic measurement (using one illumination pattern only) directly in the diffusion approximation, with no added noise.

We placed, centered at $z=5$, four spherical inhomogeneities (cf. Figure \ref{fig:example_fem_setup}) into a cubical grid $(x,y,z) \in [0,20] \times [0,10] \times [0,10]$ with resolution $320\times160\times160$. The fluence $u$ is calculated by numerically solving the PDE \eqref{eq:diff_eq} with homogeneous Dirichlet boundary conditions (simulating a uniform illumination). For this purpose, we take a self-written \emph{MATLAB} finite element solver (that splits the grid into a tetrahedral mesh and then uses linear basis elements). To get simulated initial pressure data $\H$, we re-sampled $u$ at the grid centerpoints and built $\H=\Gamma \mu u$ by multiplication with $\Gamma\mu$ (see Figure \ref{fig:example_fem_data}).
\FloatBarrier
\begin{figure}[htp]
	\centering
		\subfloat[Inhomogenities]{ \includegraphics[width=0.44\textwidth]{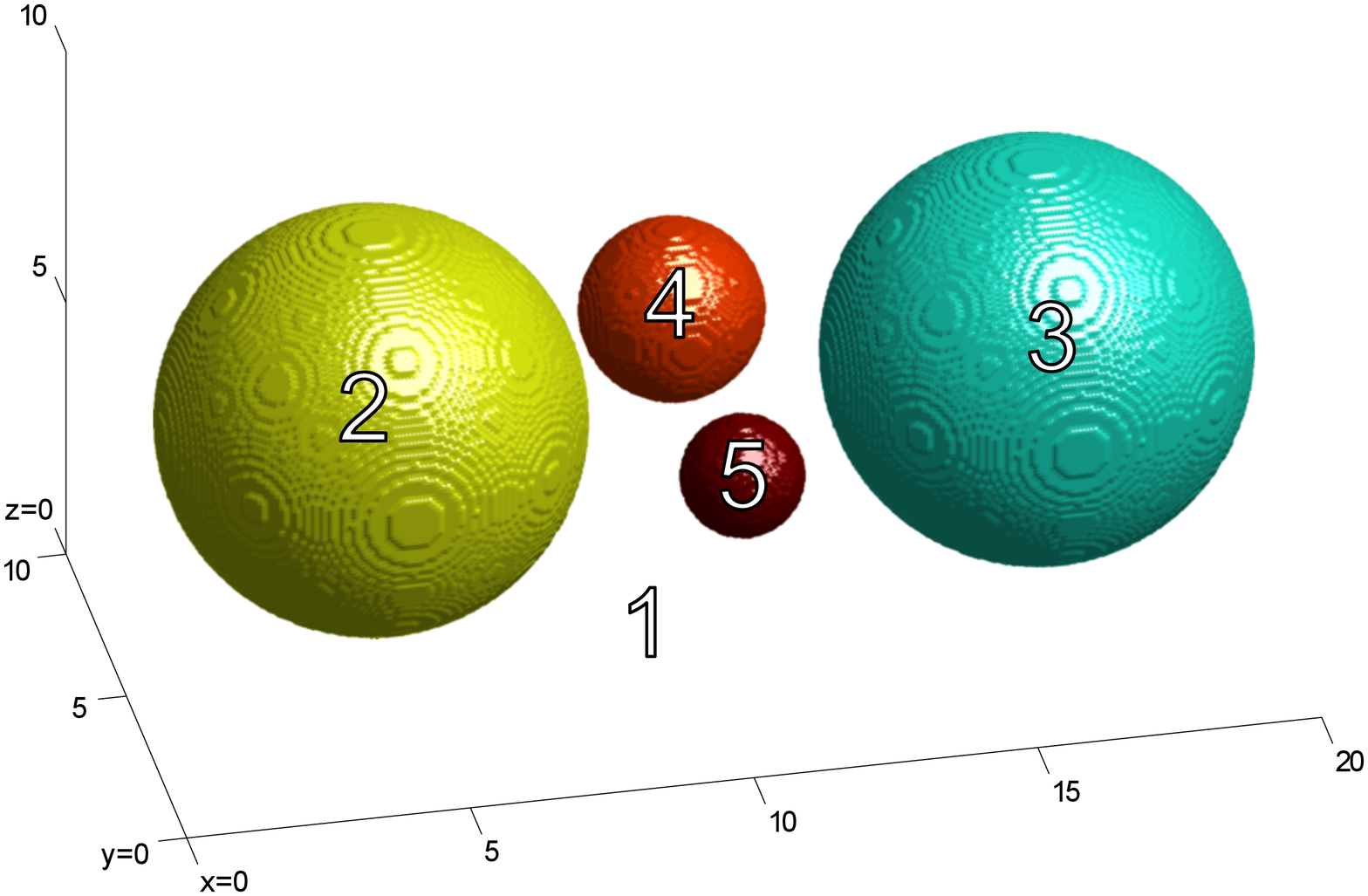}} \hspace{0.05\textwidth}
  		\subfloat[Material properties]
  		{ 
  		  \begin{tabular}[b]{cccc}	 		  	  
              & $\boldsymbol{\mu}$ & $\boldsymbol{D}$ & $\boldsymbol{\Gamma}$ \\
			  \toprule
		   	  \textbf{Region 1} & 0.1  & 1    & 1 \\
		   	  \midrule
		   	  \textbf{Region 2}   & 0.2  & 1    & 1 \\
		   	  \midrule		   	  
		   	  \textbf{Region 3}   & 0.1  & 0.25 & 1 \\
		   	  \midrule		   	  		   	  
		   	  \textbf{Region 4}   & 0.01 & 1    & 10 \\
		   	  \midrule		   	  
		   	  \textbf{Region 5}   & 1    & 10   & 0.01 \\		   	  
		   	  \bottomrule 
  	      \end{tabular} 	  
		}
	\caption[Simulation setup]
	{ Simulation setup. Spherical inhomogeneities viewed from top left (a) and their material properties (b).
	}	
	\label{fig:example_fem_setup}	
\end{figure}
\FloatBarrier
\begin{figure}[htp]
	\centering
		\subfloat[Fluence $u$]{ \includegraphics[width=0.44\textwidth]{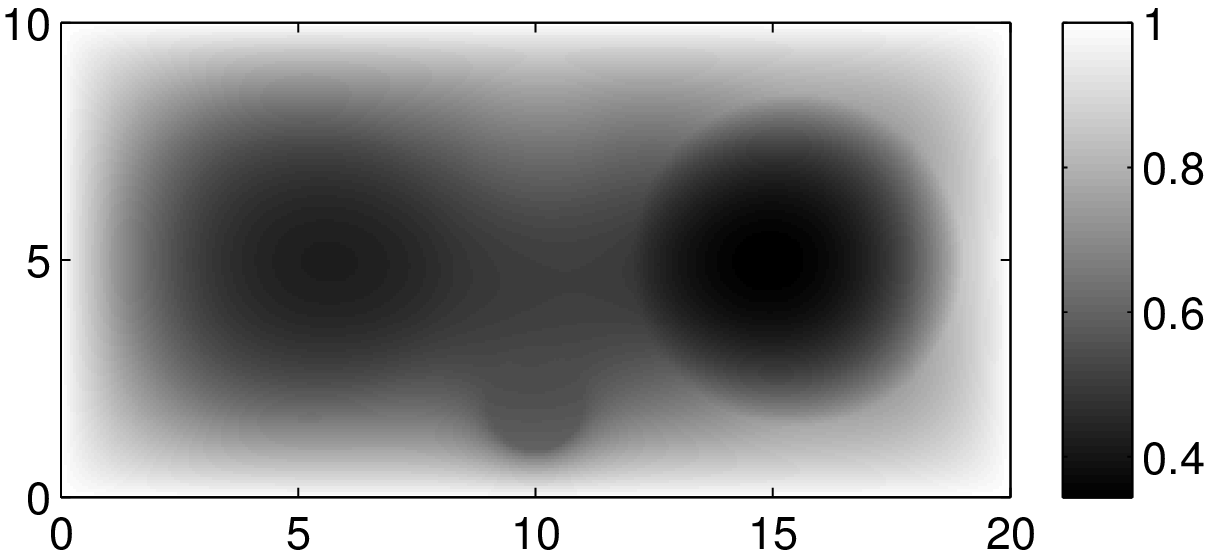}} \hspace{0.05\textwidth}
  		\subfloat[Initial pressure $\H$]{ \includegraphics[width=0.44\textwidth]{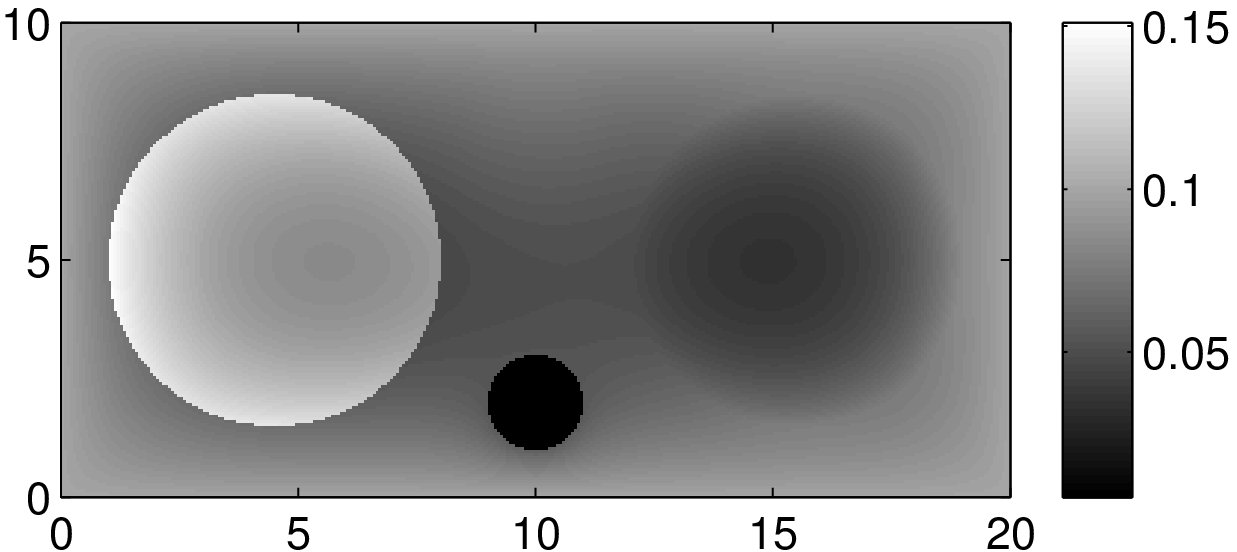}} \\
		\subfloat[$\log_{10} |\nabla \H|$]{ \includegraphics[width=0.44\textwidth]{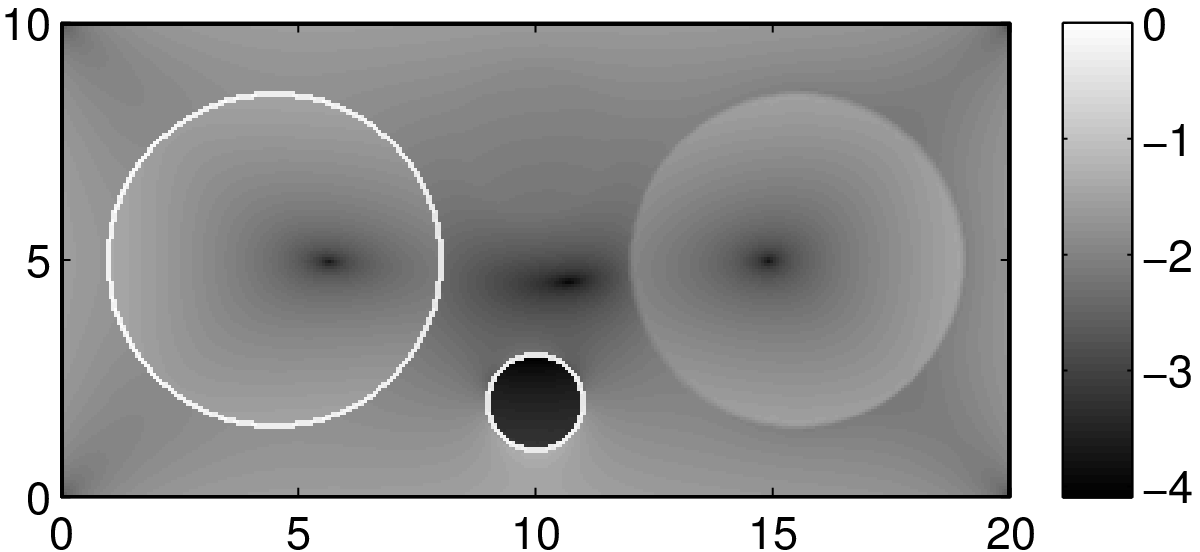}} \hspace{0.05\textwidth}
  		\subfloat[$\log_{10} |\Delta \H|$]{ \includegraphics[width=0.44\textwidth]{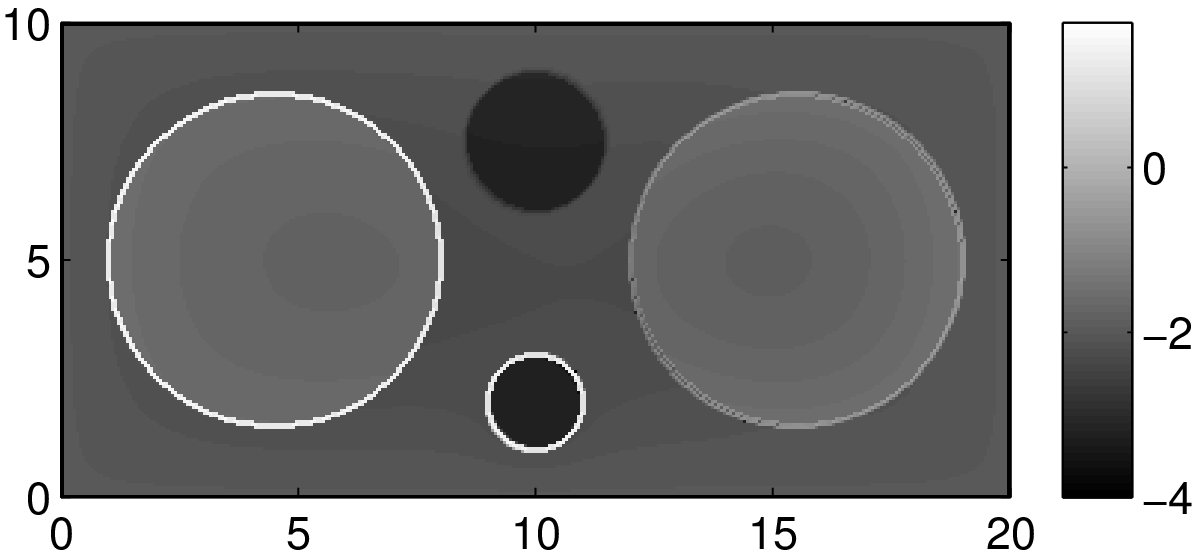}} \\
	\caption[Simulation results]
	{ FEM-simulated fluence, photoacoustic initial pressure and derivatives. The fluence is chosen to be uniform at the boundary. Derivatives are calculated by finite differences. All images are plane cuts at $z=5$.
	}	
	\label{fig:example_fem_data}	
\end{figure}

In Figure \ref{fig:example_fem_data}, one can see how the inhomogeneities affect the data $\H$ (cf. Proposition \ref{prop:jump_detection}). Spheres 1 and 4 have contrast in $\Gamma \mu$ with respect to the background, so their boundaries are are visible in $\H$. Sphere 2 displays contrast in $D$, but not in $\Gamma\mu$, its interface with the background hence can be seen in $|\nabla \H|$. Since in this particular example, the field $\nabla u$ is never parallel to the sphere's boundary, the whole boundary is visible. Sphere 3 has the same $\Gamma \mu$ and $D$ as the background, so it's only visible in $|\Delta \H|$.

\begin{figure}[htp]
	\centering
  		\subfloat[Estimated segmentation and jumps]{ \includegraphics[width=0.44\textwidth]{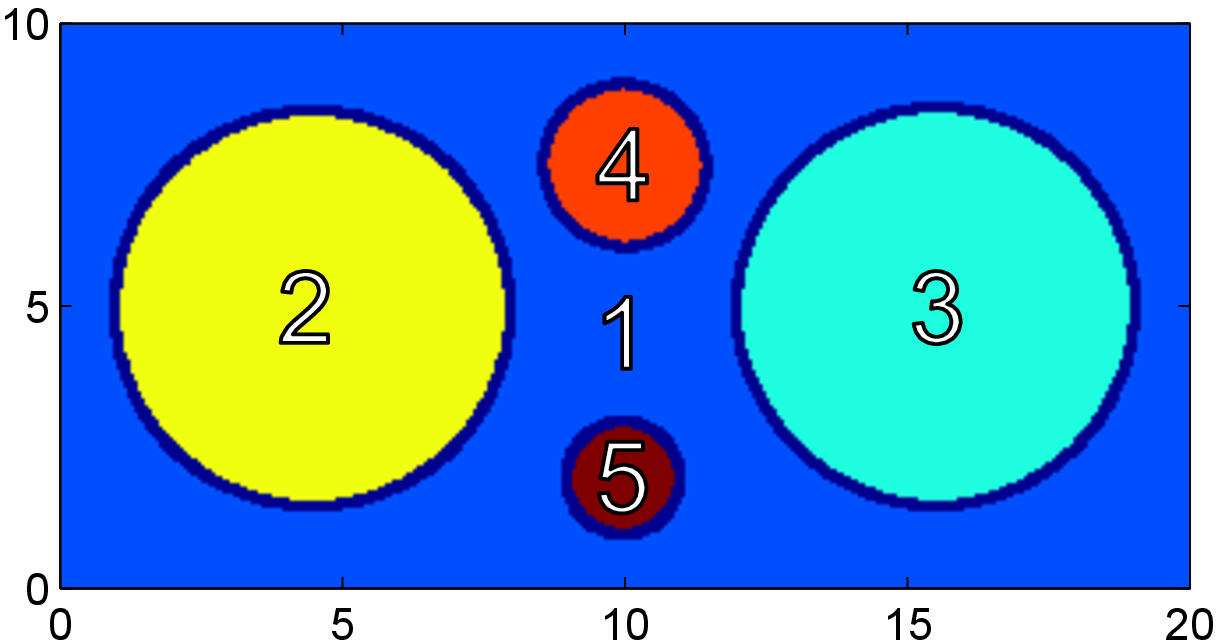}} \\
		\subfloat[Parameter estimates and relative errors]
  		{ 
  		  \begin{tabular}[b]{clll}	 		  	  
              					 & \multicolumn{1}{c}{$\boldsymbol{\hat{\mu}}$} 
              					 & \multicolumn{1}{c}{$\boldsymbol{\hat{D}}$} 
              					 & \multicolumn{1}{c}{$\boldsymbol{\hat{\Gamma}}$} \\
			  \toprule
		   	  \textbf{Region 1}  &  0.0995 (0.5\%)  &  1.0000 (0\%)    &   1.0000 (0\%)\\
		   	  \midrule		   	    
		   	  \textbf{Region 2}  &  0.1977 (1.1\%)  &  0.9880 (1.2\%)  &   1.0043 (0.4\%)\\
		   	  \midrule		   	  		   	  
		   	  \textbf{Region 3}  &  0.1105 (10.5\%) &  0.2759 (10.4\%) &   0.9072 (9.3\%)\\
		   	  \midrule		   	  
		   	  \textbf{Region 4}  &  0.0097 (2.8\%)  &  0.9723 (2.8\%)  &  10.2429 (2.4\%)\\
		   	  \midrule		   	  
		   	  \textbf{Region 5}  &  0.6238 (37.6\%) &  6.2361 (37.6\%) &   0.0158 (58.4\%)\\
		   	  \bottomrule 
  	      \end{tabular} 	  
		}  		
	\caption[Reconstruction results]
	{ Estimated jumps and regions in a plane cut at $z=5$ (a), estimated parameters and their relative errors (b).
	}	
	\label{fig:example_fem_results}	
\end{figure}

Figure \ref{fig:example_fem_results} shows the reconstruction results. As reference values, we used the values of $D$ and $\Gamma$ in the background (Region 1). All parameter discontinuities were recovered. Without noise, by far the biggest accuracy bottleneck is the estimation of jumps in $D$ from the normal components of $\nabla \H$, in particular for smaller structures (with respect to the resolution). The estimation of $\mu \Gamma$ and $\frac{\mu}{D}$ works almost perfectly for this type of data.
\FloatBarrier

\subsection{Example using Monte-Carlo-generated data}
\label{sub:monte_carlo_data}

For the second numerical example, we used \emph{MMC}, an open source 3D Monte-Carlo photon transfer simulator by Qianqian Fang (see \cite{Fan10} for details), to simulate photoacoustic measurements.

\FloatBarrier
We again placed four inhomogeneities, centered at $z=5$, into a homogeneous background cubic grid $(x,y,z) \in [0,10] \times [0,10] \times [0,10]$ with resolution $150\times150\times150$ (cf. Figure \ref{fig:example_mmc_setup}). Note that two of the structures touch (Regions 2 and 3). We deliberately chose the material parameters such that there is always enough contrast in $\Gamma\mu$ and $D$ so that edge detection in $\frac{\Delta \H}{\H}$ is not necessary (this proved to be very tricky in the presence of noise since it uses second order differences). Using \emph{MMC}, we calculated fluences $u^k, \,k=1,\ldots,6$ for $6$, for multiple sources (placed in the center of each of the cube's faces). We again re-sampled $u$ at the grid centerpoints, built initial pressure data $\H^k=\Gamma \mu u^k$ (by multiplication with $\Gamma\mu$) and added $5\%$ multiplicative Gaussian noise (which corresponds to a constant signal-to-noise ratio of about $26\,\mathrm{dB}$). 
\begin{figure}[htp]
	\centering
		\subfloat[Inhomogenities]{ \includegraphics[width=0.3\textwidth]{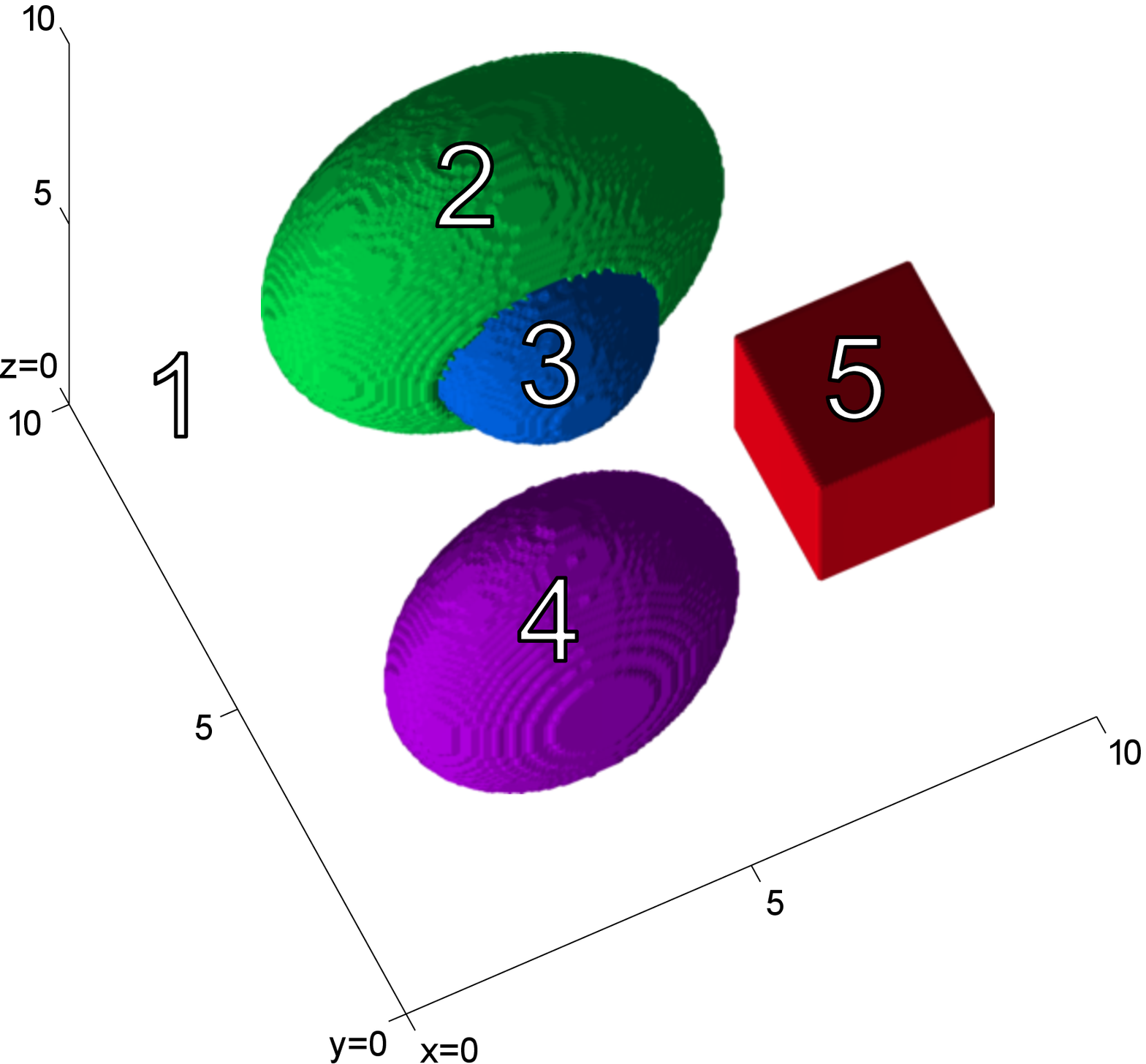}} \hspace{0.1\textwidth}
  		\subfloat[Material properties]
  		{ 
  		  \begin{tabular}[b]{cccc}	 		  	  
              & $\boldsymbol{\mu}$ & $\boldsymbol{D}$ & $\boldsymbol{\Gamma}$ \\
			  \toprule
		   	  \textbf{Region 1}   & 0.01  	& 0.166    & 1 \\
		   	  \midrule
		   	  \textbf{Region 2}   & 0.01  	& 0.056    & 1 \\
		   	  \midrule		   	  
		   	  \textbf{Region 3}   & 0.01  	& 0.166    & 1.2 \\
		   	  \midrule		   	  		   	  
		   	  \textbf{Region 4}   & 0.02 	& 0.538    & 0.5 \\
		   	  \midrule		   	  
		   	  \textbf{Region 5}   & 0.006   & 0.111    & 0.8 \\		   	  
		   	  \bottomrule 		   	  
  	      \end{tabular} 	  
		}
	\caption[Simulation setup]
	{ Simulation setup. Inhomogeneities viewed from the top right (a) and their material properties (b).
	}	
	\label{fig:example_mmc_setup}	
\end{figure}
\FloatBarrier
\begin{figure}[htp]
	\centering
		\subfloat[$\log_{10} u^1$]{ \includegraphics[width=0.25\textwidth]{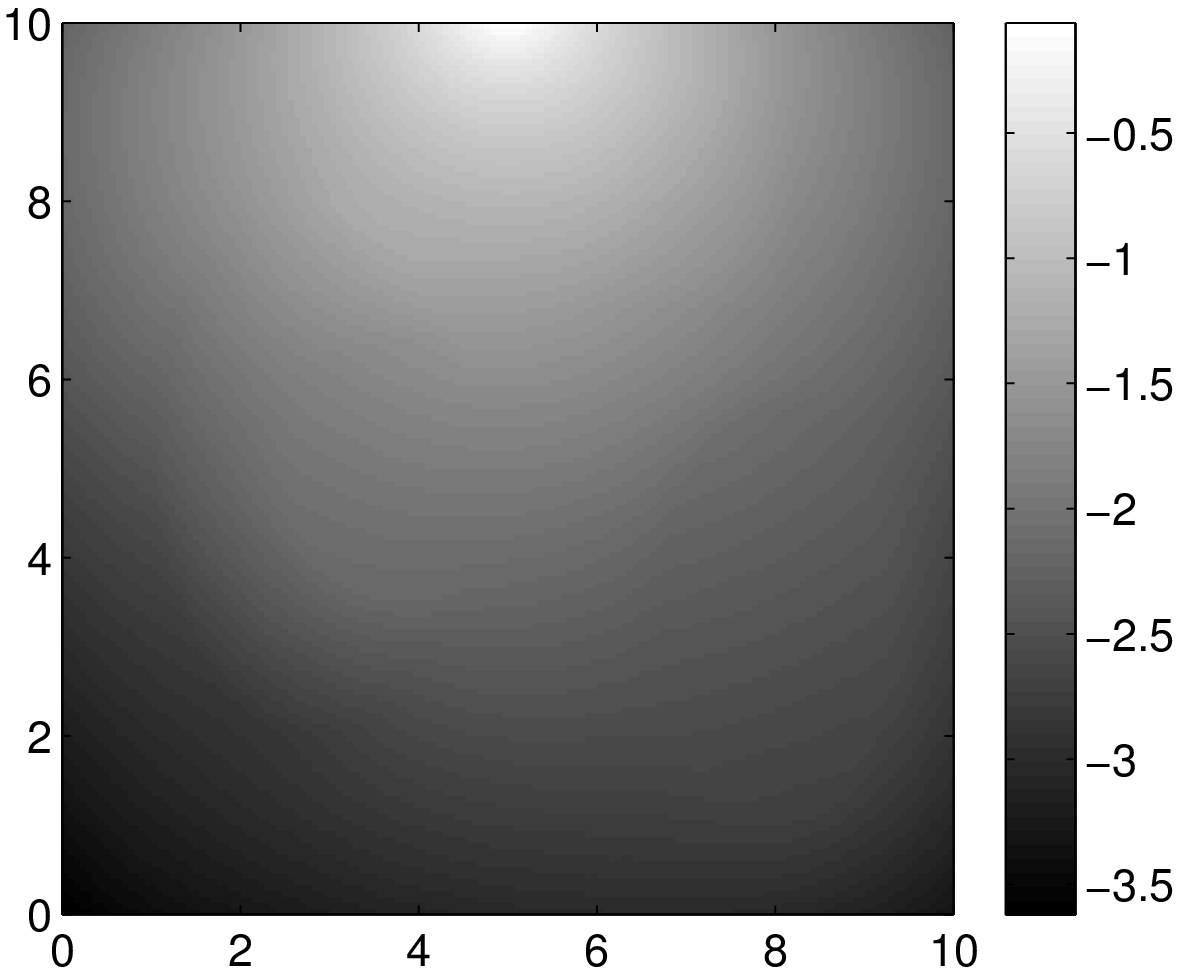}} \hspace{0.05\textwidth}
  		\subfloat[$\log_{10} \H^1$]{ \includegraphics[width=0.25\textwidth]{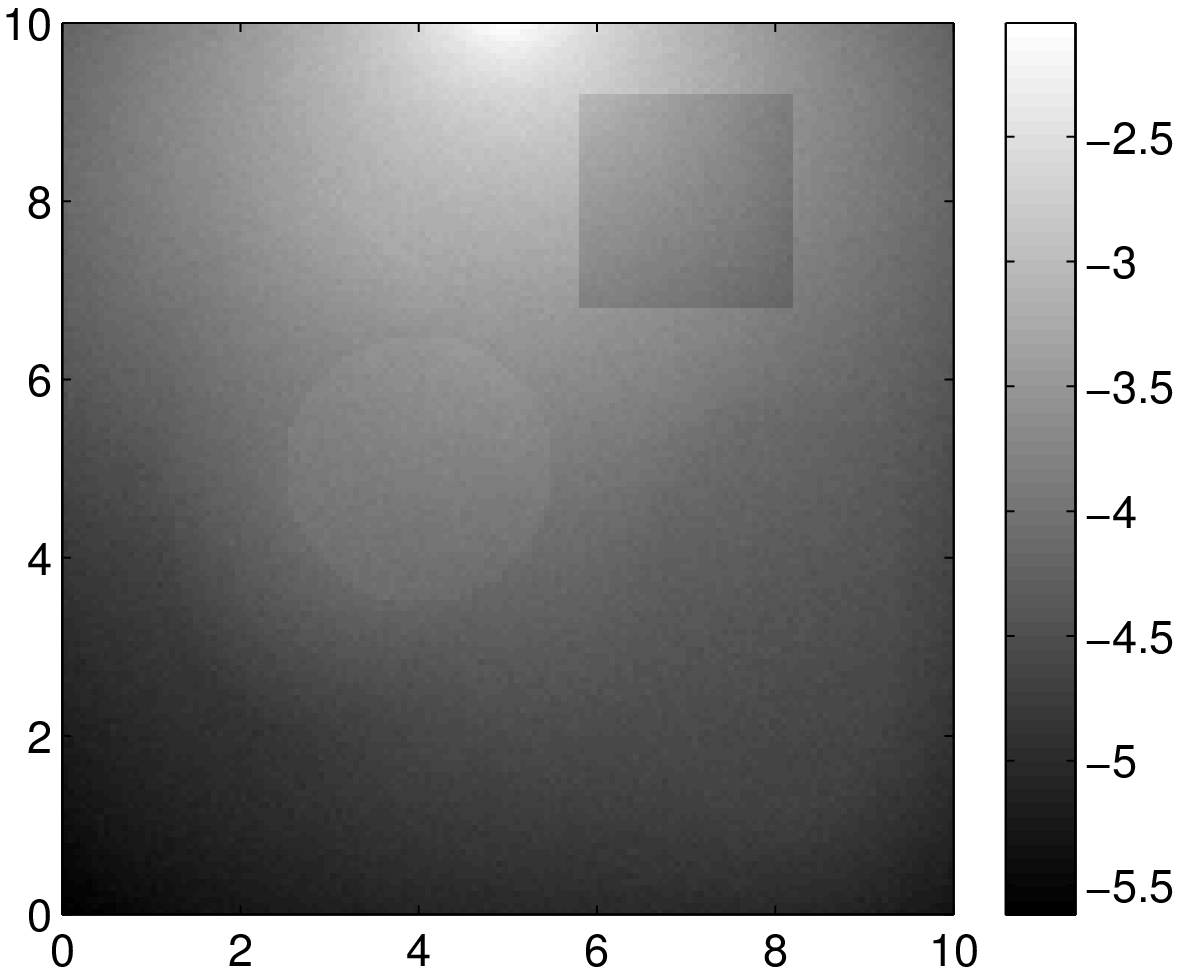}} \\
		\subfloat[$\log_{10} |\nabla \H^1| $]{ \includegraphics[width=0.25\textwidth]{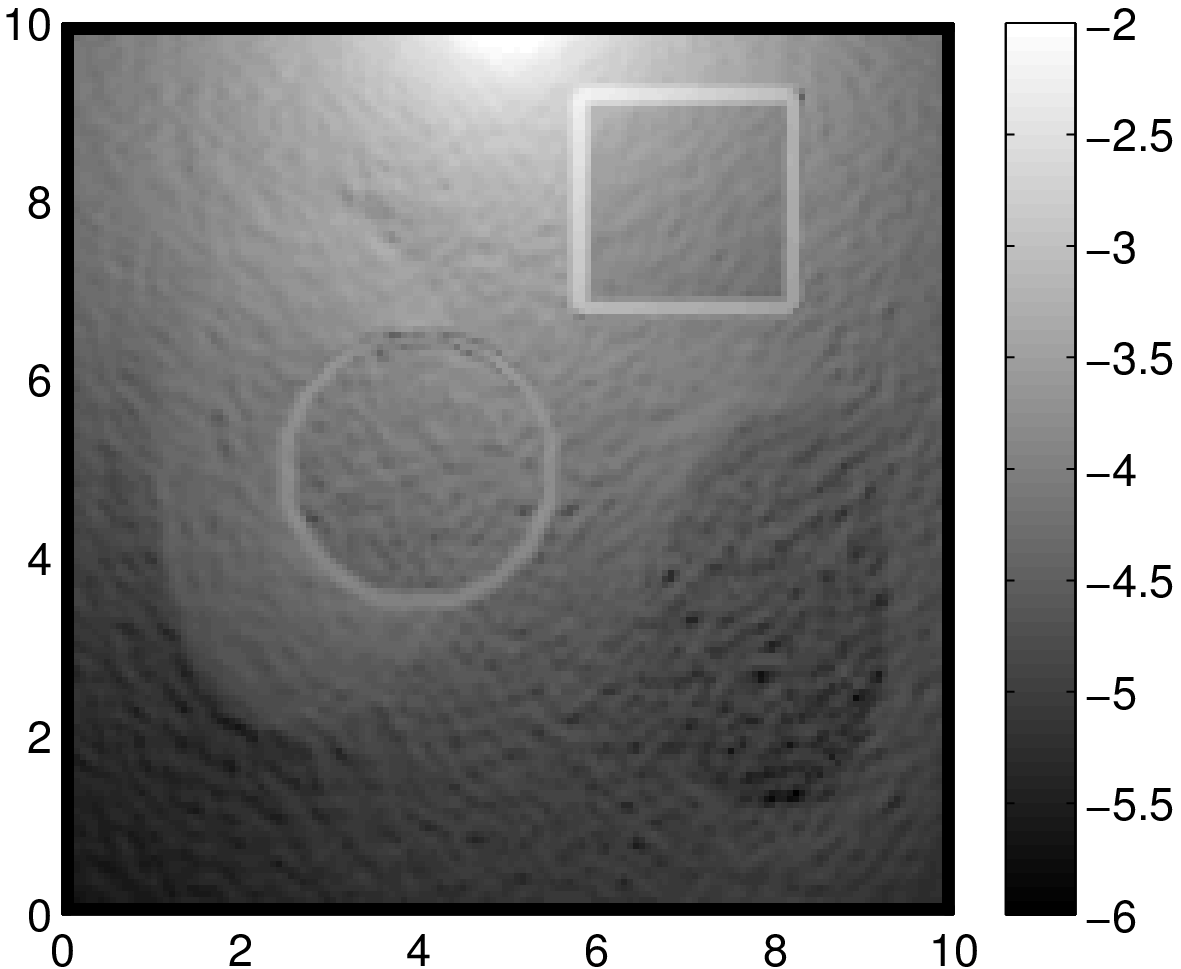}} \hspace{0.05\textwidth}
  		\subfloat[$\frac 1 6 \sum_k \log_{10} |\nabla \H^k| $]{ \includegraphics[width=0.25\textwidth]{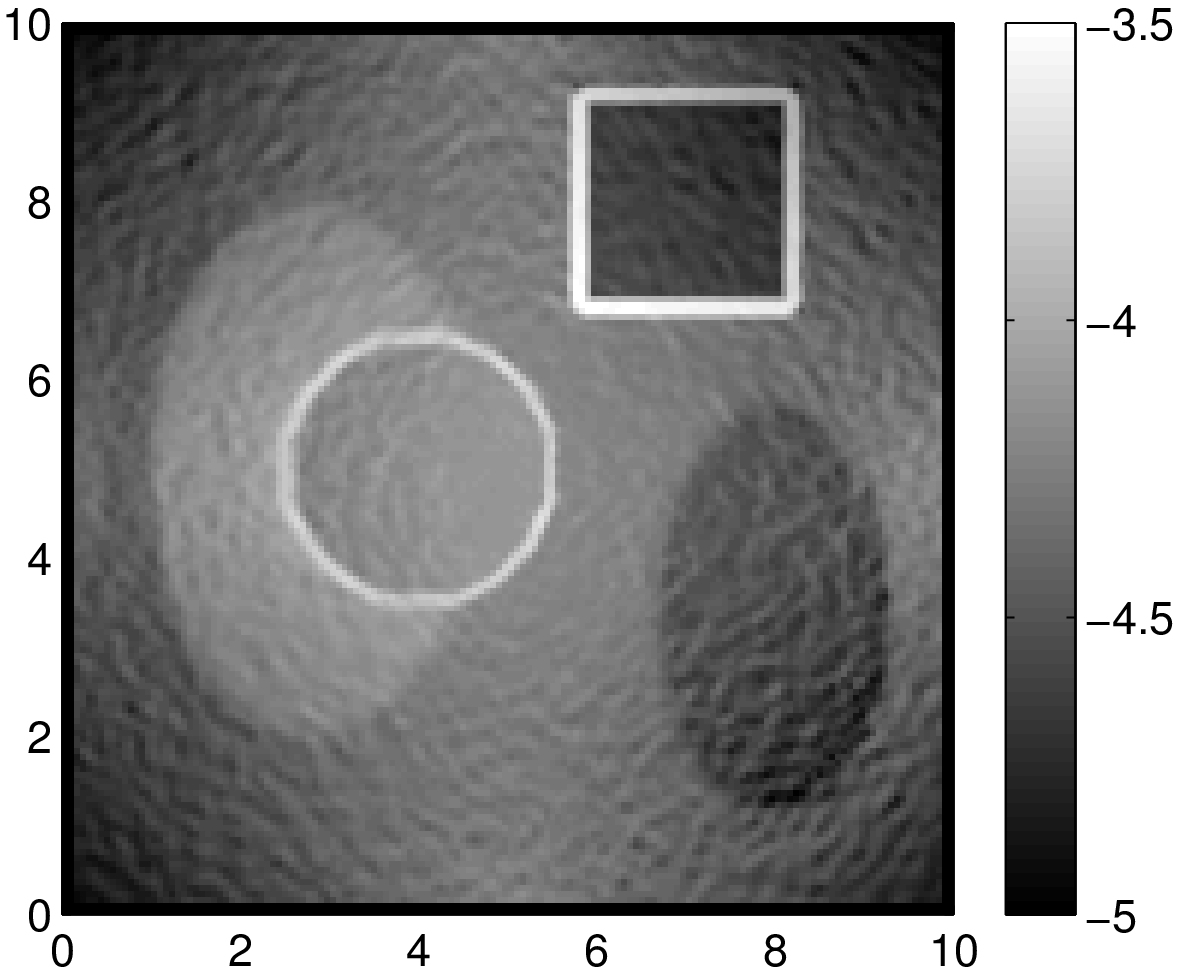}} \\
	\caption[Simulation results]
	{ MMC-simulated fluence, photoacoustic initial pressure (with noise) and its derivatives. The light source which generated $u_1$ is placed at the top of the image. Derivatives are calculated by convolution with a Gaussian (with standard deviation of $1$ pixel) followed by finite differences. All images are plane cuts at $z=5$.
	}	
	\label{fig:example_mmc_data}	
\end{figure}

Figure \ref{fig:example_mmc_data} shows a Monte-Carlo-simulated fluence $u^1$ and initial pressure $H^1$ (for which the light source at the top of the $z=5$ plane cut). Regions 3 and 5 are are visible in $\log\H^1$ due to contrast in  $\Gamma\mu$. Regions 2 and 4 appear in $\log|\nabla \H^1|$. At some parts of the regions boundaries, the $\nabla u$ is parallel to the boundary, which leads to vanishing contrast. Taking the mean of $\log |\nabla \H^k|$ (over the $6$ sources), the whole boundary is becomes visible. 
\FloatBarrier
\begin{figure}[htp]
	\centering
  		  \begin{tabular}[b]{clll}	 		  	  
              					 & \multicolumn{1}{c}{$\boldsymbol{\hat{\mu}}$} 
              					 & \multicolumn{1}{c}{$\boldsymbol{\hat{D}}$} 
              					 & \multicolumn{1}{c}{$\boldsymbol{\hat{\Gamma}}$} \\
			  \toprule
		   	  \textbf{Region 1}  &  0.012 (19.8\%)  &  0.166 (0\%)     &   1.0000 (0\%)\\
		   	  \midrule		   	    
		   	  \textbf{Region 2}  &  0.014 (41.2\%)  &  0.077 (39.2\%)  &   0.839 (16.1\%)\\
		   	  \midrule		   	  		   	  
		   	  \textbf{Region 3}  &  0.011 (11.5\%)  &  0.190 (14.7\%)  &   1.284 (7\%)\\
		   	  \midrule		   	  
		   	  \textbf{Region 4}  &  0.023 (16.2\%)  &  0.447 (16.9\%)  &   0.511 (2.1\%)\\
		   	  \midrule		   	  
		   	  \textbf{Region 5}  &  0.007 (19.2\%)  &  0.128 (15.5\%)  &   0.806 (0.8\%)\\
		   	  \bottomrule 
  	      \end{tabular} 	  
	\caption[Reconstruction results]
	{ Parameter estimates and relative errors.
	}	
	\label{fig:example_mmc_results}	
\end{figure}
\FloatBarrier
Figure \ref{fig:example_mmc_results} shows the reconstruction results. As reference values, we again used the values of $D$ and $\Gamma$ in the background. All parameter discontinuities were recovered. As before, errors in the estimation of jumps in $D$ from the normal components of $\nabla \H$ were the most significant. 
\FloatBarrier

\section{Conclusion}
\label{sec:conclusion}

Our theoretical analysis shows that in many cases (e.g., if enough measurements such that \eqref{eq:detcond} holds in the region of interest are available), unique reconstruction of piecewise constant $\mu,D,\Gamma$ from photoacoustic measurements at a single wavelength is possible. Our numerical implementation of the analytical reconstruction procedure works with reasonable accuracy, even with Monte Carlo generated data (which satisfies the diffusion approximation, which we use for reconstruction, only approximately). Our numerical method, however, requires data with very high resolution and large parameter contrast. In addition, due to the fact that we use second derivatives of the data, our  method is very sensitive to noise, so use with real data might turn out to be challenging. 

\section{Acknowledgements}
\label{sec:acknowledgements}

This work has been supported by the Austrian Science Fund (FWF) within the national research network Photoacoustic Imaging in Biology and Medicine (project S10505-N20) and by the IK I059-N funded by the University of Vienna. 

\appendix

\section{Derivation of transmission formulation}
\label{sec:transmission_cond}

In this section (following the proof in \cite{AttButMic06}), we prove that under some regularity assumptions, a function $u$ is a weak solution of 
\begin{equation}
\label{eq:dv_diff_eq}
	-\div(D \nabla u) + \mu u = 0 \quad \text{in $\Omega$}
\end{equation}
with piecewise smooth parameters $\mu,D$ if and only if 
\begin{itemize}
	\item[(1)] $u$ a classical solution in regions where the parameters are smooth,
	\item[(2)] $u$ is continuous,
	\item[(3)] the transmission condition \eqref{eq:dv_transmission_cond} holds at the jumps.
\end{itemize}

Let $\Omega \subset \mathbb{R}^n$ and $(\Omega_m)_{m=1}^M$ be piecewise-$C^1$ domains such that $\overline\Omega= \bigcup_{m=1}^M \overline\Omega_m$.

Denote by $T$ the part of the subregion boundaries that is $C^1$ and in the closure of at most two subregions. We require that the partition is chosen such $\mathcal{H}^{n-1}(\bigcup_{m=1}^M \partial\Omega_m \setminus T)=0$, i.e., the set junctions where more than three subregions meet or the boundary is not $C^1$ has zero surface measure.

Furthermore, let the parameters $\mu, D > 0$ be bounded and piecewise smooth, i.e., of the form 
\begin{equation*}
\label{eq:dv_parameters}
	\enskip \mu = \sum_{m=1}^M \mu_m 1_{\Omega_m}, \enskip D = \sum_{m=1}^M D_m 1_{\Omega_m} 
\end{equation*}
with $\mu_m, D_m \in C^{\infty}(\overline\Omega_m)$.  For a corresponding solution $u$ of \eqref{eq:dv_diff_eq}, let 
\begin{equation*}
	u_m:= u|_{\Omega_m},  \, m=1,\ldots,M.
\end{equation*}

\begin{lemma}
\label{prop:transmission_cond}
\hfill \\
Let $u$ be a weak solution of \eqref{eq:dv_diff_eq}. Furthermore, let $u_m, \, m=1,\ldots,M$ satisfy
\begin{equation}
\label{eq:dv_regularity_additional}
	u_m \in C^1(\overline\Omega_m \cap B) \quad \text{for all $B$ with $B \cap T = \emptyset$.}
\end{equation} 
Then $u \in C^\alpha(\Omega)$ for some $\alpha > 0$ and $u_m \in C^\infty(\Omega_m)$. Additionally, the restrictions $u_m$ satisfy
\begin{equation}
\label{eq:dv_strong_eq}
	    -\div( D_m \nabla u_m )+ \mu_m u_m = 0  \quad \text{in $\Omega_m$ } 
\end{equation}
and, almost everywhere on interfaces $I_{mn}=\partial\Omega_m \cap \partial\Omega_n$,
\begin{equation}
\label{eq:dv_transmission_cond}
		D_m \nabla u_m \cdot \nu = D_n \nabla u_n \cdot \nu \quad  \text{(for any interface normal $\nu$)}.
\end{equation}

\end{lemma}

\begin{proof}
A weak solution $u$ of \eqref{eq:dv_diff_eq} satisfies $u \in H^1(\Omega)$ and
\begin{equation*}
\label{eq:dv_weak_solution_subset}
	\int_\Omega D \nabla u \cdot \nabla \phi + \mu u \phi \,dx = 0	\text{ for all } \phi \in C^\infty_c(\Omega).
\end{equation*}

Since the equation is elliptic, we have $u_m \in C^\infty(\Omega_m)$ by interior regularity \cite[Corollary 8.11]{GilTru01} and hence, from integration by parts, $\int_\Omega -\div(D_m \nabla u_m) \phi + \mu_m u_m \phi \,dx= 0$ for all $\phi \in C^\infty_c(\Omega_m)$, which shows that \eqref{eq:dv_strong_eq} holds classically in $\Omega_m$, $m=1,\ldots,M$.

From De Giorgi-Nash-Moser theorem \cite[Theorem 8.22]{GilTru01} we get $u \in C^\alpha(\overline\Omega)$ for some $\alpha > 0$. 

Next, let $I_{mn}=\partial\Omega_m \cap \partial\Omega_n$ be the interface between some $\Omega_m$ and $\Omega_n$. For almost all $x \in I_{mn}$ (those in $T$), there exists an open ball $B \Subset \Omega$ such that $x \in B = B_m \cup B_n = (\Omega_m \cap B) \cup (\Omega_n \cap B)$ (by the restriction on the partition).

Using integration by parts and \eqref{eq:dv_strong_eq} we get for all $\phi \in C^\infty_c(B) \subset C^\infty_c(\Omega)$
\begin{equation*}
\begin{aligned}
	0 &= \int_\Omega D \nabla u \cdot \nabla \phi + \mu u \phi \,dx \\
	&= \int_{\Omega_m} D_m \nabla u_m \cdot \nabla \phi + \mu_m u_m \phi \,dx + \int_{\Omega_n} D_n \nabla u_n \cdot \nabla \phi + \mu_n u_n \phi \,dx \\ 
	&= \int_{\partial B_m} (D_m \nabla u_m \cdot \nu) \phi + \int_{\partial B_n} (D_n \nabla u_n \cdot \nu) \phi \,dS \\
	&= \int_{I_{mn} \cap B} (D_m \nabla u_m \cdot \nu - D_n \nabla u_n \cdot \nu) \phi \,dS.
\end{aligned}
\end{equation*}
The transmission condition \eqref{eq:dv_transmission_cond} follows since $\nabla u_m, \nabla u_n \in C(I_{mn} \cap B)$ (by assumption \eqref{eq:dv_regularity_additional}). 

\end{proof}

For certain partition geometries, weak solutions of \eqref{eq:dv_diff_eq} always satisfy condition \eqref{eq:dv_regularity_additional}. For instance, Li and Nirenberg \cite[Proposition 1.4]{LiNir03} showed that if $(\Omega_m)_{m=2}^M$ are inclusions with smooth boundaries (which may also touch in some points) and background $\Omega_1$, one gets $u_m \in C^\infty(\overline\Omega_m)$. 

For sufficiently regular $u_m$ (e.g., in the setting just described) we can also derive the converse of Lemma \ref{prop:transmission_cond}: 
\begin{lemma}
\label{prop:transmission_cond_converse} Let $u_m:= u|_{\Omega_m} \in C^2(\overline\Omega_m), \,m=1,\ldots,M,$ satisfy
\begin{equation}
\label{eq:dv_strong_eq_converse}
	    -\div( D_m \nabla u_m )+ \mu_m u_m = 0 \quad \text{in $\Omega_m$ } \\   
\end{equation}
and, on interfaces $I_{mn}=\partial\Omega_m \cap \partial\Omega_n$ with normal $\nu$,
\begin{equation}
\label{eq:dv_transmission_cond_converse}
\begin{aligned}
		u_m &= u_n \\
		D_m \nabla u_m \cdot \nu &= D_n \nabla u_n \cdot \nu.
\end{aligned}		
\end{equation}
Then $u$ is a weak solution of \eqref{eq:dv_diff_eq}.
\end{lemma}
\begin{proof}
To get $u \in H^1(\Omega)$, we first show that the weak gradient of $u$ is given by 
\begin{equation}
\label{eq:dv_weak_gradient}
	\nabla u = \sum_{m=1}^M \nabla u_m 1_{\Omega_m}.
\end{equation}

To see that, note that for all $\phi \in C^\infty_c(\Omega)$
\begin{equation*}
 -\int_\Omega u \nabla \phi \,dx = -\sum_{m=1}^M \int_{\Omega_m} u_m \nabla \phi \,dx = \sum_{m=1}^M \left(\int_{\Omega_m} \nabla u_m \phi \,dx - \int_{\partial\Omega_m} u_m \phi\nu \,dS \right).
\end{equation*}

The interior boundary terms cancel out due to $u_m = u_n$, the exterior boundary terms vanish since $\supp \phi \Subset \Omega$), so the weak gradient of $u$ is given by \eqref{eq:dv_weak_gradient}. 
Hence $u \in H^1(\Omega)$ since 
\begin{equation*}
\begin{aligned}
	\norm{u}_{H^1(\Omega)}^2 &= \int_\Omega |u|^2 + |\nabla u|^2 \,dx = \sum_{m=1}^M \left( \int_{\Omega_m} |u_m|^2 + \left|\nabla u_m \right|^2 \,dx \right) \\
	&= \sum_{m=1}^M \norm{u_m}_{H^1(\Omega_m)}^2 \leq C \sum_{m=1}^M \norm{u_m}_{W^{1,\infty}(\overline\Omega_m)}^2 < \infty
\end{aligned}
\end{equation*}

Furthermore, using integration by parts, \eqref{eq:dv_strong_eq_converse} and \eqref{eq:dv_transmission_cond_converse} imply
\begin{equation*}
\begin{aligned}
	\int_{\Omega} D \nabla u \cdot  \nabla \phi + \mu u \phi \,dx  &= \sum\limits_{m=1}^M \int_{\Omega_m} D_m \nabla u_m \cdot \nabla \phi + \mu_m u_m \phi \,dx \\ 
	&= \sum_{m=1}^M \int_{\partial\Omega_m} D_m (\nabla u_m \cdot \nu) \phi \,dS = 0
\end{aligned}
\end{equation*}

for $\phi \in C^\infty_c(\Omega)$ since the boundary terms cancel out due to \eqref{eq:dv_transmission_cond_converse} and $\supp \phi \Subset \Omega$, so $u$ is a weak solution of \eqref{eq:dv_diff_eq}.
\end{proof}

\section{Differential Canny edge detection}
\label{sec:canny}

In differential \emph{Canny edge detection} as proposed by Lindeberg (cf. \cite{Lin98}), one starts from a \emph{scale space} representation $f_\sigma = f * g_\sigma$ of a two-dimensional image $f\colon \mathbb{R}^2 \to \mathbb{R}$, where $g_\sigma$ is a Gaussian kernel with standard deviation $\sigma$. Edges at scale $\sigma$ are then \emph{defined} (with finite resolution, no natural notion of discontinuity exists) as local maxima of the gradient magnitude $|\nabla f_\sigma|$ in gradient direction $\nabla f_\sigma$. Additionally, it is proposed to additionally maximize a certain functional measuring edge strength in scale space (which allows for automatic scale selection).

We want to use a similar algorithm to find the discontinuities of a three-dimensional function $f\colon \mathbb{R}^3 \to \mathbb{R}$ (which will be $\log \H^k$, $\log |\nabla \H^k|$ or $\log|\frac{\Delta \H^k}{\H^k}|$). Jumps of $f$ that are sufficiently big compared its continuous variation (within a grid step) lead to sudden changes of intensity (above some threshold) in the corresponding finite-resolution image. Heuristically, we have a similar situation as in Canny edge detection. That is, jump surfaces approximately correspond to thresholded maxima of $|\nabla f_\sigma|$ in gradient direction, where $f_\sigma = f * g_\sigma$ is the scale-space representation of $f$ for a properly chosen scale $\sigma$ (for simplicity, we will work at a single, manually chosen scale in this paper).

To estimate the jump set, we thus have to solve for \emph{fixed} $\sigma$ and $v=\nabla f_\sigma$ 
\begin{equation}
\label{eq:canny_eqn}
	\begin{aligned}
		\partial_v |\nabla f_\sigma|^2 &= \sum_{i,j=1}^3 v_i v_j\,\partial_{x_i x_j} f_\sigma = 0\\
		\partial_{vv} |\nabla f_\sigma|^2 &= \sum_{i,j,k=1}^3 v_i v_j v_k\,\partial_{x_i x_j x_k} f_\sigma > 0.
	\end{aligned}
\end{equation}

For discrete (voxelized) $f$, the solution manifold can be calculated with sub-voxel resolution. To restrict $E$, the solution surface of \eqref{eq:canny_eqn}, to parts where the gradient magnitude (and thus also the jump across the surface) is large enough, we perform \emph{hysteresis thresholding}. That is, we first apply a lower threshold $\rho_1$ to the jump strength $|\nabla f_\sigma |$ to get 
\begin{equation*}
	E_1 = \{ x \in E \ \big| \ |\nabla f_\sigma(x) | \geq \rho_1 \}.
\end{equation*}

Then, we remove all connected components $C \subset E_1$ for which the jump strength is never above a higher threshold $\rho_2$, so we get our final jump set $E_2$ with
\begin{equation*}
	E_2 = \bigcup \{ C \subset E_1 \ \big| \ C \text{ is connected} \ \wedge \ \exists x \in C\colon \ |\nabla f_\sigma(x) | \geq \rho_2 \}.
\end{equation*}

As a final step, we remove all isolated structures smaller than a certain size (which are usually due to misdetections and too small for further processing).

\def\cprime{$'$}
  \providecommand{\noopsort}[1]{}\def\ocirc#1{\ifmmode\setbox0=\hbox{$#1$}\dimen0=\ht0
  \advance\dimen0 by1pt\rlap{\hbox to\wd0{\hss\raise\dimen0
  \hbox{\hskip.2em$\scriptscriptstyle\circ$}\hss}}#1\else {\accent"17 #1}\fi}
  \def\cprime{$'$}

\end{document}